\documentclass[11pt]{article}
\usepackage{amssymb,amsthm,amsmath,amsfonts}
\usepackage[latin1]{inputenc}
\usepackage{graphicx}
\usepackage{verbatim}
\usepackage{array}
\usepackage{color}
\newtheorem{Theorem}{Theorem}[section]

\newtheorem{Corollary}[Theorem]{Corollary}

\newtheorem{Lemma}[Theorem]{Lemma}

\newtheorem{Remark}[Theorem]{Remark}

\newtheorem{Proposition}[Theorem]{Proposition}

\newtheorem{Claim}[Theorem]{Claim}

\begin{document}

\title{Ground state solutions for a nonlocal equation in $\mathbb{R}^2$ involving vanishing potentials and exponential critical growth
\thanks{Research partially supported by CAPES Cod 001 and CNPq grant 308735/2016-1.}}
\author{Francisco S. B. Albuquerque \\ Universidade Estadual da Paraíba,
Departamento de Matemática, \\ CEP 58700-070, Campina Grande-PB, Brazil  \\
\textsf{fsiberio@cct.uepb.edu.br}
\\ \\ Marcelo C. Ferreira \\ Universidade Federal de Campina Grande,
Unidade Acadêmica de Matemática, \\ CEP 58000-000, Campina Grande-PB, Brazil  \\
\textsf{marcelo@mat.ufcg.edu.br}
 \\ \\ Uberlandio B. Severo\thanks{Author to whom any correspondence should be addressed.} \\ Universidade Federal da Para\'{\i}ba, Departamento de Matem\'atica, \\
 CEP 58051-900, Jo\~ao Pessoa-PB, Brazil \\
 \textsf{uberlandio@mat.ufpb.br} }

\date{}

\maketitle{}

\numberwithin{equation}{section}

\begin{abstract}
In this paper, we study the following class of nonlinear equations:
$$
-\Delta u+V(x) u = \left[|x|^{-\mu}*(Q(x)F(u))\right]Q(x)f(u),\quad x\in\mathbb{R}^2,
$$
where $V$ and $Q$ are continuous potentials, which can be unbounded or vanishing at infintiy, $f(s)$ is a continuous function, $F(s)$ is the primitive of $f(s)$, $*$ is the convolution operator and $0<\mu<2$.  Assuming that the nonlinearity $f(s)$ has exponential critical growth, we establish the existence of ground state solutions by using variational methods. For this, we prove a new version of the Trudinger-Moser inequality for our setting, which was necessary to obtain our main results.
\end{abstract}

\bigskip

\noindent{\small\emph{2010 Mathematics Subject Classification:}
35J15, 35J25, 35J60.

}

\section{Introduction}\label{}
This paper is concerned with the existence of solutions for nonlinear equations of the form
\begin{equation}\label{P}
-\Delta u+V(|x|) u = \left[|x|^{-\mu}*(Q(|x|)F(u))\right]Q(|x|)f(u),\quad x\in\mathbb{R}^2,
\end{equation}
where the nonlinear term $f$ is allowed to satisfy the exponential growth by mean of the Trudinger-Moser inequality, $F$ denotes its primitive, the radial potentials $V$ and $Q$ can be unbounded, singular at the origin or decaying to zero at infinity and $0<\mu<2$. Here,
 $|x|^{-\mu}*(Q(|x|)F(u))$ denotes the convolution between the Riesz potential $|x|^{-\mu}$ and $Q(|\cdot|)F(u(\cdot))$.

The study of Eq. \eqref{P} is in part motivated by works concerning the equation
\begin{equation}\label{choquard1}
-\Delta u+V(x)u = \left(|x|^{-\mu}*|u|^{p}\right)|u|^{p-2}u,\quad x\in\mathbb{R}^N,
\end{equation}
where $N\geq 3$, $0<\mu<N$, $V:\mathbb{R}^{N}\rightarrow\mathbb{R}$ is a continuous potential and $p\geq 2$. Eq. (\ref{choquard1}) is generally named as \textit{Choquard equation} or \textit{Hartree type equation} and appears in various physical contexts. For example, in the case $N=3$, $V(x)=1$, $p=2$ and $\mu=2$, Eq. \eqref{choquard1} first appeared in the seminal work by S. I. Pekar \cite{Pekar} describing the quantum mechanics of a polaron at rest. As mentioned by Lieb in \cite{Lieb1} and under the same case, Ph. Choquard used Eq. \eqref{choquard1} to model an electron trapped in its own hole, as a certain approximation to Hartree-Fock theory of one-component plasma (see also \cite{Lieb3} for more physical background of \eqref{choquard1}). We point out that Eq. \eqref{choquard1} is also known as the \textit{Schrödinger-Newton equation}, see \cite{BGDB,CSV,MPT,pen,TM}. In \cite{Lieb1}, Lieb proved that the ground state solution of Eq. \eqref{choquard1} is radial and unique up to translations with $\mu= 1$, $p = 2$ and $V$ is a positive constant. Later, in \cite{Lions}, Lions showed the existence of a sequence of radially symmetric solutions. In \cite{CCS,Ma-Zhao,Moroz-Scha}, the authors showed regularity, positivity, radial symmetry and decay estimates at infinity of ground states solutions as well.

In the past two decades, Eq. \eqref{choquard1} has attracted a lot of interest due to the appearance of convolution type nonlinearities. Many authors have used the variational methods to investigate the nonlinear Choquard equation like \eqref{choquard1} and involving general classes of nonlinearities, we refer the readers to \cite{acker,AFY,BJS,CS,GVS,LSY,luo,MVS,SGY} for the study of existence and multiplicity of different types of solutions and \cite{AY0,AY0.5,MVS2} for the study of existence and concentration behavior of solutions. For the convenience of the reader, we suggest \cite{MJS3} for a good review of the Choquard equation.

In dimension two, the case is very special and quite delicate, because as we know for bounded domains $\Omega\subset\mathbb{R}^{2}$, the Sobolev embedding theorem assures that $H_0^{1}(\Omega)\hookrightarrow L^{q}(\Omega)$ for any $q\in[1,+\infty)$, but  $H_0^{1}(\Omega)\not\hookrightarrow L^{\infty}(\Omega)$. Therefore, to fill this gap, the Trudinger-Moser inequality (which is stated in \eqref{classical-TM} below) cames as a replacement of the Sobolev inequality in the limiting case. 

In order to address problems with exponential critical growth, one of the most important tools is the Trudinger-Moser inequality (see \cite{Moser,Trudinger}), which says that if $\Omega$ is a bounded domain in $\mathbb{R}^{2}$, then for all $\alpha>0$ and $u\in H_0^{1}(\Omega)$, it holds $e^{\alpha u^{2}}\in L^{1}(\Omega)$. Moreover
\begin{equation}\label{classical-TM}
\sup_{\|\nabla u\|_{L^2(\Omega)}\leq1}\int_{\Omega}e^{\alpha u^{2}}\,\mathrm{d}x \left\{
\begin{aligned}
<\infty, \ \ & \textrm{when $\alpha\leq 4\pi$}\\
=\infty, \ \ & \textrm{when $\alpha>4\pi$.}
\end{aligned}
\right.
\end{equation}
 In the whole space $\mathbb{R}^2$, the following version of Trudinger-Moser
inequality was proved in  \cite{DOO97} (see also \cite{Cao} for an equivalent version):
\[
e^{\alpha|u|^2} -1 \in L^1(\mathbb{R}^2)\quad\mbox{for all}\quad u\in
H^{1}(\mathbb{R}^2)\quad\mbox{and}\quad\alpha >0.
\]
Moreover, if $\alpha < 4\pi$ and $|u|_{L^2(\mathbb{R}^2)}\leq M$,
there exists a constant $C = C(M,\alpha)$ such that
$$
\sup_{\|\nabla u\|_{L^2(\mathbb{R}^2)}\leq 1}\int_{\mathbb{R}^2}
\left(e^{\alpha|u|^2} - 1\right)\, dx \leq C.
$$
 We point out that, in recent years, many generalizations, extensions and applications of these results have been made. We refer to \cite{djr} for a general discussion on problems involving critical growth of Trudinger-Moser type.

As far as we know, the nonlocal Choquard type equation involving critical exponential growth in $\mathbb{R}^2$ was first studied in \cite{ACTY,AY1}, where the authors considered the existence of ground state solution for the following critical nonlocal equation with periodic potential:
\begin{equation}\label{choquard2}
-\Delta u+W(x) u = \left(|x|^{-\mu}*F(u)\right)f(u),\quad x\in\mathbb{R}^2.
\end{equation}
Under the set of assumptions on the functions $W$ and $f$ below:
\begin{itemize}
    \item[$(W1)$] $W(x)\geq W_0$ for all $x\in \mathbb{R}^{2}$ and for some $W_0>0$;

	 \item[$(W2)$] $W(x)$  is a $1$-periodic continuous function;
  
    \item[$(f^1)$] i) $f(s)=0$ for all $s\leq 0$ and $0\leq f(s)\leq e^{4\pi s^{2}}$ for all $s\geq 0$; and\\
    ii) there exist $s_0>0$, $M_0>0$ and $q\in (0,1]$ such that 
   $$
   0<s^qF(s)\leq M_0f(s),\quad\forall s\geq s_0;
   $$
    
     \item[$(f^2)$] there exit $p>\frac{2-\mu}{2}$ and $C_p>0$ such that $f(s)\sim C_p s^{p}$ as $s\rightarrow 0$;
     
      \item[$(f^3)$] there exists $\theta>1$ such that 
    $$
    \theta\int_0^{s}f(t)\,\mathrm{d}t=\theta F(s)\leq f(s)s,\quad\forall s\geq 0;
    $$
      
       \item[$(f^4)$] there exists $\beta_0>0$ large enough such that
	$$
	\lim_{s\rightarrow +\infty}\frac{sf(s)F(s)}{e^{8\pi s^2}}\geq \beta_0,
	$$ 
\end{itemize}
they obtained the existence of ground state solution in $H^{1}(\mathbb{R}^{2})$ for Eq. \eqref{choquard2}. The goal of the present paper is to continue the study of nonlocal equations in $\mathbb{R}^{2}$ as in \eqref{choquard2}, by considering potentials $W(x)$ which can be singular at the origin and vanishing at infinity. Moreover, we deal with nonlinearities of the form $[|x|^{-\mu}*(Q(x)F(u))]Q(x)f(u)$, where $Q(x)$ is a weight function and $f$ behaves at infinity like $e^{\alpha s^{2}}$ for some $\alpha>0$. 

Here, due to the presence of the weight function $Q$ in the nonlocal term, which can be singular at the origin, we need to derive a new version of the Trudinger-Moser inequality for our context (see Proposition \ref{TM} in Section 2). 

In this work, we impose the following hypotheses on the potential $V$ and the weight $Q$:
\begin{itemize}
    \item[$(V0)$] $V\in C(0,\infty)$, $V(r)>0$ and there exist $a_0>-2$ and $a>-2$ such that
$$
\limsup_{r\rightarrow 0^{+}}\dfrac{V(r)}{r^{a_0}} < \infty \quad \mathrm{and}\quad \liminf_{r\rightarrow+\infty}\dfrac{V(r)}{r^a}>0;
$$
\item[$(Q0)$] $Q\in C(0,\infty)$, $Q(r)>0$ and there exist $b_0>-\frac{4-\mu}{2}$ and $b<\frac{a(4-\mu)}{4}$ such that
$$
\limsup_{r\rightarrow0^{+}}\frac{Q(r)}{r^{b_0}}<\infty\quad\text{and}\quad
\limsup_{r\rightarrow+\infty}\frac{Q(r)}{r^b}<\infty.
$$
\end{itemize}
Hereafter, we say that $(V,Q)\in \mathcal{K}$ if $(V0)$ and $(Q0)$ hold. According to the features of the functions $V$ and $Q$, our approach becomes more delicate when compared with the works \cite{ACTY,AY1}. In this direction, we complement and improve the results in \cite{ACTY,AY1}, in the sense that we are leading with potentials can be singular at the origin and vanishing at infinity.

In this work, we are also interested in the case where the nonlinear term $f(s)$ has maximal growth on $s$ which allows us to treat problem~\eqref{P} variationally. Explicitly, in view of the classical Trudinger-Moser inequality \eqref{classical-TM} and following \cite{Ad-Yadava,djairo-olimpio-ruf}, we say that a function $f(s)$ has $\alpha_0$-\emph{exponential critical growth} at $+\infty$ if
\begin{equation}\label{def.cresc.critico}
\lim_{s\rightarrow+\infty}\dfrac{f(s)}{e^{\alpha s^2}}=0,\ \ \forall\alpha>\alpha_0,\ \ \text{and}\ \lim_{s\rightarrow+\infty}\dfrac{f(s)}{e^{\alpha s^2}}=+\infty,\ \ \forall\alpha<\alpha_0.
\end{equation}
Similarly we define exponential critical growth at $-\infty$. Throughout this paper, the following hypotheses on $f(s)$ will be imposed:
\begin{itemize}
   \item[$(f_1)$] $f:\mathbb{R}_+\rightarrow\mathbb{R}$ is continuous and $\lim_{s\rightarrow 0^{+}}f(s)/s^{\frac{2-\mu}{2}}=0$;
    \item[$(f_2)$] there exists $\theta>1$ such that 
    $$
    \theta\int_0^{s}f(t)\,\mathrm{d}t=\theta F(s)\leq f(s)s,\quad\forall s\geq 0;
    $$
    \item[$(f_3)$] there exist $q>1$ and $\xi>0$ such that 
   $$
   F(s)\geq\xi s^q,\quad\forall s\in [0,1].
   $$
\end{itemize}

In order to state our main results, we need to introduce some
notations. $L^{s}(\mathbb{R}^{2})$, for $1\leq s<\infty$, denotes the Lebesgue space with the norms
$$
|u|_{s}=\left(\int_{\mathbb{R}^2}|u|^s\,\mathrm{d}x\right)^{1/s}
$$
and we shall use the notation $\|u\|_{L^{s}(\Omega)}$ for the norm in the Lebesgue space $L^{s}(\Omega)$, for $1\leq s<\infty$ and any $\Omega\subset\mathbb{R}^{2}$. If $1\leq q<\infty$ we define the weighted Lebesgue spaces
$$
L^q(\mathbb{R}^2;Q):=\left\{u:\mathbb{R}^2\rightarrow\mathbb{R} : u \; \text{is measurable and} \int_{\mathbb{R}^2}Q(|x|)|u|^q\,\mathrm{
d}x<\infty\right\} 
$$ 
and
$$ 
L^2(\mathbb{R}^2;V):=\left\{u:\mathbb{R}^2\rightarrow\mathbb{R} : u \; \text{is measurable and} \int_{\mathbb{R}^2}V(|x|)u^2\,\mathrm{d}x<\infty\right\}, 
$$
endowed, respectively, with the norms
$$
|u|_{L^{q}(\mathbb{R}^{2};Q)}=\left(\int_{\mathbb{R}^2}Q(|x|)|u|^q\,\mathrm{d}x\right)^{\frac{1}{q}}\quad\textrm{and}\quad|u|_{L^{2}(\mathbb{R}^{2};V)}=\left(\int_{\mathbb{R}^2}V(|x|)u^2\,\mathrm{d}x\right)^{\frac{1}{2}}.
$$
We also define the functional space
$$
Y:=\left\{u\in L_{\mathrm{loc}}^{2}(\mathbb{R}^{2}):|\nabla u|\in  L^{2}(\mathbb{R}^{2})\ \ \text{and}\ \ \int_{\mathbb{R}^2}V(|x|)u^2\,\mathrm{d}x<\infty\right\}
$$
endowed with the norm $\|u\|:=\sqrt{\langle u,u\rangle}$ induced by the scalar product
\begin{equation}\label{inner-def}
\langle u,v\rangle:=\int_{\mathbb{R}^2}\left(\nabla u\cdot\nabla v+V(|x|)uv\right)\mathrm{d}x.
\end{equation}
The proof that $(Y, \| \cdot\|)$ is a Hilbert space is not direct. Thus, in the next section, we present its proof. Furthermore, the subspace
$$
Y_{\mathrm{rad}}:
=\{u\in Y:u\ \text{is radial}\}
$$
is closed in $Y$ and thus it is a Hilbert space itself.

Let $C^{\infty}_0(\mathbb{R}^2)$ be the set of smooth functions with compact support. We say that $u:\mathbb{R}^{2}\rightarrow\mathbb{R}$ is a weak solution for \eqref{P} if $u\in Y$ and it holds the equality
\begin{equation}\label{weak-sol-def}
\int_{\mathbb{R}^2}(\nabla u\cdot\nabla\phi+V(|x|)u\phi)\,\mathrm{d}x-\int_{\mathbb{R}^2}\left[|x|^{-\mu}*(Q(|x|)F(u))\right]Q(|x|)f(u)\phi\,\mathrm{d}x=0,
\end{equation}
for all $\phi\in C^{\infty}_0(\mathbb{R}^2)$.
Now, our main results read as follow.

\begin{Theorem}\label{main-thm}
Assume that $0<\mu<2,\ (V,Q)\in \mathcal{K}$ and $f(s)$ satisfies $\eqref{def.cresc.critico},\ (f_1)-(f_3)$ with $\xi>0$, given in $(f_3)$, verifying
$$
\xi\geq\max\left\{\xi_1, \left[\frac{\frac{\|Q\|_{L^1(B_{1/2})}^2}{2}(q-1)\left(\frac{\xi_1^2}{q}\right)^{q/(q-1)}}{\frac{4-\mu}{\alpha_0}\left(1+\frac{2b_0}{4-\mu}\right)\frac{\pi(\theta-1)}{2\theta}}\right]^{(q-1)/2}\right\},
$$
where
$$
\xi_1:=\frac{[\pi+\|V\|_{L^1(B_1)}]^{\frac{1}{2}}}{\|Q\|_{L^1(B_{1/2})}}.
$$
Then, Eq. \eqref{P} has a nontrivial weak solution in $Y_{\mathrm{rad}}$.
\end{Theorem}

\begin{Theorem}\label{ground-state}
Under the conditions of Theorem \ref{main-thm} and supposing that $f(s)/s$ is increasing for $s>0$, then the solution obtained in Theorem \ref{main-thm} is a ground state.
\end{Theorem}

For our second existence result, we replace condition $(f_3)$ by the following conditions:
\begin{itemize}
	\item[$(f_4)$] there exist $s_0>0$, $M_0>0$ and $\vartheta\in (0,1]$ such that 
	$$
	0<s^{\vartheta}F(s)\leq M_0f(s),\quad\forall s\geq s_0;
	$$
	\item[$(f_5)$] $\displaystyle\liminf_{s\rightarrow +\infty}\frac{F(s)}{e^{\alpha_0s^2}}=:\beta_0>0.
	$  
\end{itemize}

\begin{Theorem}\label{main-thm2}
	Assume that $0<\mu<2,\ (V,Q)\in \mathcal{K}$ and $f(s)$ satisfies $\eqref{def.cresc.critico}$, $(f_1),\ (f_2),\ (f_4)$ and $(f_5)$. If we also assume that $ \displaystyle\liminf_{r \to 0^+} Q(r)/r^{b_0} > 0$, then Eq. \eqref{P} has a nontrivial weak solution in $Y_{\mathrm{rad}}$.
\end{Theorem}
We emphasize that conditions $(f_3)$ and $(f_5)$ are usually used to estimate the minimax level of the energy functional associated to the problem and, therefore, we study two situations. The second one, by using $(f_5)$, is more delicate and we do not require a constraint on the positive constant $\beta_0$.

To conclude this introduction, we would like to say that compared with the local equation
$$
-\Delta u+V(x) u = Q(x)f(u),\quad x\in\mathbb{R}^2,
$$ 
our main difficulty is that there exists the nonlocal term $|x|^{-\mu}*(Q(x)F(u))$ in \eqref{P}, which leads to more complex computations, causes some mathematical difficulties and makes the problem rough and particularly interesting.


The outline of the paper is as follows: in Section 2, we show some useful preliminary results which will be used later on. In Section 3, we establish our variational setting. In Section 4, we check the geometric conditions and we prove some properties on the Palais-Smale sequences of the associated functional. Moreover, we get a more precise information about the minimax level obtained by the mountain-pass theorem. In Section 5, we prove Theorems \ref{main-thm} and \ref{ground-state}. The final section offers the proof of Theorem \ref{main-thm2}.


\section{Some useful preliminaries}

In this section, we introduce some useful results that will be used in our proofs. We start by proving that $Y$ is a Hilbert space. Hereafter, positive constants (possibly different) will be denoted by $C,C_0,C_1,C_2,\cdots$ and to indicate that a constant depends on other quantities we list them in parentheses: $C(\ldots)$. Furthermore, we denote the open ball centered at the origin with radius $R>0$ by $B_R$.
\begin{Proposition}
	The space $(Y,\|\cdot\|)$ is a Hilbert space.
\end{Proposition}
\begin{proof}
	Let $(u_n)\subset Y$ be a Cauchy sequence. Hence, $(\frac{\partial u_n}{\partial x_i}),\, i=1,2,$ and $(V^{1/2}u_n)$ are both Cauchy sequences in $L^{2}(\mathbb{R}^{2})$. Thus,
	\begin{equation}\label{banach1}
	\frac{\partial u_n}{\partial x_i}\rightarrow u^i\ (i=1,2)\ \ \text{and}\ \ V^{1/2}u_n\rightarrow v\ \ \text{in}\ L^{2}(\mathbb{R}^{2}),
	\end{equation}
	as $n\rightarrow+\infty$. So, up to subsequence,
	\begin{equation}\label{banach2}
	\frac{\partial u_n}{\partial x_i}\rightarrow u^i\ (i=1,2)\ \ \text{and}\ \ u_n\rightarrow V^{-1/2}v=:w\ \ \text{a.e. in}\ \mathbb{R}^{2},
	\end{equation}
	as $n\rightarrow+\infty$. Notice that 
	$$
	\int_{\mathbb{R}^2}V(|x|)w^2\,\mathrm{d}x=\int_{\mathbb{R}^2}v^2\,\mathrm{d}x<\infty.
	$$
	On the other hand, for each $R>0$, let $\varphi $ be in $C_0^\infty(\mathbb{R}^2)$ satisfying $\textrm{supp}(\varphi)\subset B_{R+1}$ and $\varphi\equiv 1$ in $B_R$. Thus, by the Poincaré inequality we reach
	\begin{equation}\label{Patos}
	\begin{aligned}
	\int_{B_R}|u_n-u_m|^{2}\,\mathrm{d}x&\leq \int_{B_{R+1}}|\varphi(u_n-u_m)|^{2}\,\mathrm{d}x\leq C_1\int_{B_{R+1}}|\nabla (\varphi (u_n-u_m))|^2\,\mathrm{d}x\\
	 &\leq C_2\int_{B_{R+1}\backslash B_R}|\nabla \varphi|^2|u_n-u_m|^2\,\mathrm{d}x +\int_{B_{R+1}}| \varphi|^2|\nabla u_n-\nabla u_m|^2\,\mathrm{d}x\\
	 &\leq \frac{C_2\|\nabla \varphi\|_\infty^2}{M_R}\int_{B_{R+1}\backslash B_R}V(|x|)|u_n-u_m|^2\,\mathrm{d}x +\|\varphi\|^2_\infty\int_{B_{R+1}}|\nabla u_n-\nabla u_m|^2\,\mathrm{d}x\\
	 &\leq C(R)\|u_n-u_m\|^{2},
	\end{aligned}
	\end{equation}
	where $M_R=\min_{x\in B_{R+1}\backslash B_R}V(|x|)>0$, and so $(u_n)$ is also a Cauchy sequence in $L^{2}(B_R)$. Hence, for each $R>0$, there exists $u_R\in L^{2}(B_R)$ such that
	\begin{equation}\label{banach3}
	u_n\rightarrow u_R\ \ \text{in}\ L^{2}(B_R)\quad \mathrm{and}\quad
	u_n\rightarrow u_R\ \ \text{a.e. in}\ B_R,
	\end{equation}
	as $n\rightarrow+\infty$. Due to \eqref{banach2} and \eqref{banach3}, we conclude that $u_R=w$ in $B_R$. Now, let $\varphi\in C_{0}^{\infty}(\mathbb{R}^{2})$ and $R>0$ such that $\mathrm{supp}(\varphi)\subset B_R$. For each $n\in\mathbb{N}$, we have
	$$
	\int_{\mathbb{R}^2} u_n\frac{\partial \varphi}{\partial x_i}\,\mathrm{d}x=-\int_{\mathbb{R}^2}\frac{\partial u_n}{\partial x_i}\varphi\,\mathrm{d}x,\ \ i=1,2.
	$$
	From \eqref{banach1} and \eqref{banach3} and since $u_R=w$ in $B_R$, we infer that
	$$
	\int_{\mathbb{R}^2} w\frac{\partial \varphi}{\partial x_i}\,\mathrm{d}x=-\int_{\mathbb{R}^2}u^i\varphi\,\mathrm{d}x,\ \ i=1,2.
	$$
	Hence, $w$ has weak derivative and $\partial w/\partial x_i=u^i\in L^{2}(\mathbb{R}^{2}),\ i=1,2$, which implies that $|\nabla w|\in L^{2}(\mathbb{R}^{2})$. Furthermore, from \eqref{banach1}, $\|\nabla w-\nabla u_n\|_2^{2}\rightarrow 0$ as $n\rightarrow+\infty$. Now, since $u_R=w$ in $B_R$, we deduce that $w\in L_{\mathrm{loc}}^{2}(\mathbb{R}^{2})$ and, again from \eqref{banach1},
	$$
	\int_{\mathbb{R}^2}V(|x|)|u_n-w|^2\,\mathrm{d}x=\int_{\mathbb{R}^2}\left|V^{1/2}(|x|)u_n-v\right|^2\mathrm{d}x\rightarrow 0,\ \text{as}\ n\rightarrow+\infty,
	$$
	concluding that $w\in Y$ and $u_n\rightarrow w$ in $Y$. This finishes the proof of the proposition.
\end{proof}

\begin{Remark}
	Equivalently, we can show that the functional space $Y$ can be regarded as the completion of $C_{0}^{\infty}(\mathbb{R}^{2})$ under the norm $\|u\|:=\sqrt{\langle u,u\rangle}$ (see \eqref{inner-def}). 
\end{Remark} 

\begin{Remark}\label{tec0}
	By the estimates obtained in \eqref{Patos}, we observe that, for each open ball $B_R\subset\mathbb{R}^{2}$, the space $Y$ is continuously immersed in $H^{1}(B_R)$, because by the definition of $Y$ it is sufficient to verify that if $u\in Y$, then $u\in L^{2}(B_R)$. Thus, in particular, $Y$ is continuously immersed in $L^{q}(B_R)$ for all $q\geq 1$.
\end{Remark}

Next, we recall a variant of the well-known Radial Lemma (see~\cite{Wang1}) due to W.~A. Strauss~\cite{Strauss}. Hereafter,  $B_R\setminus B_r$ denotes the annulus with interior radius $r$ and exterior radius $R$. For any set $A\subset\mathbb{R}^2$, $A^c$ denotes the complement of $A$.

\begin{Lemma}\label{lema-radial}
Suppose that condition $(V0)$ holds. Then, there exist $R_0>0$ and $C>0$ such
that for all $u\in Y_{\mathrm{rad}}$
$$
|u(x)|\leq C\|u\||x|^{-\frac{a+2}{4}},\quad\text{for all}\ \ |x|\geq R_0.
$$
\end{Lemma}

We need the following embedding result:

\begin{Lemma}\label{imersao-principal}
Suppose that $(V,Q)\in \mathcal{K}$. Then the embeddings $Y_{\mathrm{rad}}\hookrightarrow L^q(\mathbb{R}^2;Q)$ are compact for all $2\leq q<\infty$.
\end{Lemma}

\begin{Lemma}\label{tec1}
Suppose that $(V,Q)\in \mathcal{K}$ and let $0<r<1/2$ and $R\geq R_0>1$, where $R_0$ is given by Lemma \ref{lema-radial}.  Then, for all $q\geq \frac{4-\mu}{2}$, it holds
\begin{align}\label{estimativa1}
&\int_{\mathbb{R}^{2}}Q(|x|)^{\frac{4}{4-\mu}}|u|^{\frac{4q}{4-\mu}}\,\mathrm{d}x\nonumber\\
&\leq C\left(r^{\frac{4b_0}{4-\mu}+\frac{2}{\sigma}}\|u\|^{\frac{4q}{4-\mu}}+\int_{B_R\setminus B_{r}}|u|^{\frac{4q}{4-\mu}}\,\mathrm{d}x+
R^{\frac{4b}{4-\mu}-a-\left(\frac{a+2}{2}\right)\left(\frac{4q}{4-\mu}-2\right)}\|u\|^{\frac{4q}{4-\mu}}\right),
\end{align}
for all $u\in Y_{\mathrm{rad}}$ and for some $C>0$ and $\sigma>1$ be such that $b_0\sigma>-\frac{4-\mu}{2}$.
\end{Lemma}

\begin{proof}
Let $R\geq R_0>1$. By the hypotheses $(V0)$ and $(Q0)$, there exists $C_0>0$ such that
\begin{equation}\label{V-Q-foradabola}
Q(|x|)\leq C_0|x|^{b}\ \ \textrm{and}\ \ V(|x|)\geq C_0|x|^{a},\quad\textrm{for all}\ |x|\geq R.
\end{equation}
Next, we estimate the integral on the complement of the ball $B_R$. From Lemma \ref{lema-radial} and \eqref{V-Q-foradabola} we have
\begin{align*}
\int_{B_R^{c}}Q(|x|)^{\frac{4}{4-\mu}}|u|^{\frac{4q}{4-\mu}}\,\mathrm{d}x&\leq C_1\int_{B_R^{c}}|x|^{\frac{4b}{4-\mu}}|u|^{\frac{4q}{4-\mu}}\,\mathrm{d}x\\
&=\frac{C_1}{C_0}\int_{B_R^{c}}|x|^{\frac{4b}{4-\mu}}|x|^{-a}C_0|x|^{a}u^{2}|u|^{\frac{4q}{4-\mu}-2}\,\mathrm{d}x\\
&\leq\frac{C_1}{C_0}\int_{B_R^{c}}|x|^{\frac{4b}{4-\mu}-a}V(|x|)u^{2}\left(C|x|^{\left(-\frac{a+2}{2}\right)\left(\frac{4q}{4-\mu}-2\right)}\|u\|^{\frac{4q}{4-\mu}-2}\right)\mathrm{d}x\\
&=C_2\|u\|^{\frac{4q}{4-\mu}-2}\int_{B_R^{c}}|x|^{\frac{4b}{4-\mu}-a-\left(\frac{a+2}{2}\right)\left(\frac{4q}{4-\mu}-2\right)}V(|x|)u^{2}\,\mathrm{d}x.
\end{align*}
Since $a>-2$, $b<\frac{a(4-\mu)}{4}$ and $q\geq \frac{4-\mu}{2}$, we have that $\frac{4b}{4-\mu}-a-\left(\frac{a+2}{2}\right)\left(\frac{4q}{4-\mu}-2\right)<0$. Thus, we get
\begin{equation}\label{est1}
\int_{B_R^{c}}Q(|x|)^{\frac{4}{4-\mu}}|u|^{\frac{4q}{4-\mu}}\,\mathrm{d}x\leq C_2 R^{\frac{4b}{4-\mu}-a-\left(\frac{a+2}{2}\right)\left(\frac{4q}{4-\mu}-2\right)}\|u\|^{\frac{4q}{4-\mu}}.
\end{equation}
On the other hand, again by hypothesis $(Q0)$, there exists $C_0>0$ such that
\begin{equation}\label{Q-dentrodabola}
Q(|x|)\leq C_0|x|^{b_0},\quad\textrm{for all}\ 0<|x|<r_0.
\end{equation}
Now we shall estimate the integral on the ball $B_r$. For that aim, consider $\sigma>1$ such that $b_0\sigma>-\frac{4-\mu}{2}$ and a radial cut-off function $\varphi\in C_0^{\infty}(\mathbb{R}^{2})$ verifying
$$ 
0\leq \varphi\leq 1\ \text{in}\ \mathbb{R}^{2},\ \varphi\equiv 1\ \text{in}\ B_{1/2},\ \textrm{supp}(\varphi)\subset B_1\ \ \text{and}\ \ |\nabla\varphi|\leq 1\ \text{in}\ \mathbb{R}^{2}. 
$$
By \eqref{Q-dentrodabola}, Hölder's inequality and the continuous embedding $H_0^{1}(B_{1})\hookrightarrow L^{s}(B_{1})$, for all $s\geq 1$, we have
\begin{align}\label{est2}
\int_{B_r}Q(|x|)^{\frac{4}{4-\mu}}|u|^{\frac{4q}{4-\mu}}\,\mathrm{d}x&\leq C_3\int_{B_{r}}|x|^{\frac{4b_0}{4-\mu}}|\varphi u|^{\frac{4q}{4-\mu}}\,\mathrm{d}x\nonumber\\
&\leq C_3\left(\int_{B_{r}}|x|^{\frac{4b_0\sigma}{4-\mu}}\,\mathrm{d}x\right)^{\frac{1}{\sigma}}\left(\int_{B_{r}}|\varphi u|^{\frac{4q\sigma}{(4-\mu)(\sigma-1)}}\,\mathrm{d}x\right)^{\frac{\sigma-1}{\sigma}}\nonumber\\
&\leq C_4r^{\frac{4b_0}{4-\mu}+\frac{2}{\sigma}}\left(\int_{B_1}|\varphi u|^{\frac{4q\sigma}{(4-\mu)(\sigma-1)}}\,\mathrm{d}x\right)^{\frac{\sigma-1}{\sigma}}\nonumber\\
&\leq C_5r^{\frac{4b_0}{4-\mu}+\frac{2}{\sigma}}\left(\int_{B_1}|\nabla(\varphi u)|^{2}\,\mathrm{d}x\right)^{\frac{2q}{4-\mu}}\nonumber\\
&=C_5r^{\frac{4b_0}{4-\mu}+\frac{2}{\sigma}}\left(\int_{B_1\setminus B_{1/2}}|u\nabla\varphi|^{2}\,\mathrm{d}x+\int_{B_1}|\varphi\nabla u)|^{2}\,\mathrm{d}x\right)^{\frac{2q}{4-\mu}}\nonumber\\
&\leq C_6r^{\frac{4b_0}{4-\mu}+\frac{2}{\sigma}}\left(\int_{B_{1}\setminus B_{1/2}}\frac{C'}{C'}u^{2}\,\mathrm{d}x+\int_{B_1}|\nabla u|^{2}\,\mathrm{d}x\right)^{\frac{2q}{4-\mu}}\nonumber\\
&\leq C_7r^{\frac{4b_0}{4-\mu}+\frac{2}{\sigma}}\left(\int_{B_{1}\setminus B_{1/2}}V(|x|)u^{2}\,\mathrm{d}x+\int_{B_1}|\nabla u|^{2}\,\mathrm{d}x\right)^{\frac{2q}{4-\mu}}\nonumber\\
&\leq C_8r^{\frac{4b_0}{4-\mu}+\frac{2}{\sigma}}\|u\|^{\frac{4q}{4-\sigma}},
\end{align}
where $C'=\min_{1/2\leq t\leq 1}V(t)$. Next we shall estimate the integral on $B_R\setminus B_{r}$. Denoting by $M_{r,R}=\displaystyle\max_{r\leq t\leq R} Q(t)$, we get
\begin{align}\label{est3}
\int_{B_R\setminus B_{r}}Q(|x|)^{\frac{4}{4-\mu}}|u|^{\frac{4q}{4-\mu}}\,\mathrm{d}x\leq M_{r,R}^{\frac{4}{4-\mu}} \int_{B_R\setminus B_{r}}|u|^{\frac{4q}{4-\mu}}\,\mathrm{d}x.
\end{align}
Hence, joining the estimates \eqref{est1}, \eqref{est2} and \eqref{est3} the result readily follows considering an appropriate positive constant $C$.
\end{proof}

\begin{Corollary}\label{tec2}
Under the conditions of Lemma \ref{tec1}, if $u_n\rightharpoonup u$ weakly in $Y_{\mathrm{rad}}$ then
$$
\int_{\mathbb{R}^{2}}Q(|x|)^{\frac{4}{4-\mu}}|u_n-u|^{\frac{4q}{4-\mu}}\,\mathrm{d}x\rightarrow 0,\,\,\, \textrm{as}\,\,\, n\rightarrow+\infty.
$$
\end{Corollary}
\begin{proof}
Given any $\varepsilon>0$, fix $R>R_0$ large enough and $0<r\leq 1/2$ small enough such that
$$
r^{\frac{4b_0}{4-\mu}+\frac{2}{\sigma}}<\varepsilon\quad\text{and}\quad R^{\frac{4b}{4-\mu}-a-\left(\frac{a+2}{2}\right)\left(\frac{4q}{4-\mu}-2\right)}<\varepsilon.
$$
Since $u_n\rightharpoonup u$ weakly in $Y_{\mathrm{rad}}$, we have that $\|u_n-u\|\leq C_1$ for all $n\in\mathbb{N}$ and, by Lemma \ref{tec0},
$$
\int_{B_R\setminus B_{r}}|u_n-u|^{\frac{4q}{4-\mu}}\,\mathrm{d}x\rightarrow 0,
$$
as $n\rightarrow+\infty$. Hence, by estimate \eqref{estimativa1}
$$
\int_{\mathbb{R}^{2}}Q(|x|)^{\frac{4}{4-\mu}}|u_n-u|^{\frac{4q}{4-\mu}}\,\mathrm{d}x\leq C\left(\varepsilon C_1^{\frac{4q}{4-\mu}}+o_n(1)+\varepsilon C_1^{\frac{4q}{4-\mu}}\right),
$$
and so
$$
\limsup_{n\rightarrow+\infty}\int_{\mathbb{R}^{2}}Q(|x|)^{\frac{4}{4-\mu}}|u_n-u|^{\frac{4q}{4-\mu}}\,\mathrm{d}x\leq \varepsilon C\left(C_1^{\frac{4q}{4-\mu}} + C_1^{\frac{4q}{4-\mu}}\right)
$$
which shows the result.
\end{proof}

Another consequence of Lemma \ref{tec1} and \eqref{est3} is the following:
\begin{Corollary}\label{tec3}
Under the conditions of Lemma \ref{tec1}, there exists $C>0$ such that, for each $q\geq \frac{4-\mu}{2}$, we have
$$
\int_{\mathbb{R}^{2}}Q(|x|)^{\frac{4}{4-\mu}}|u|^{\frac{4q}{4-\mu}}\,\mathrm{d}x\leq C\|u\|^{\frac{4q}{4-\mu}},
$$
for all $u\in Y_{\mathrm{rad}}$.
\end{Corollary}

Inspired by the papers~\cite{Ad-Sand, Cao, Calanchi, Moser,Ruf1,Trudinger} and in order to study equation~\eqref{P}, we establish a new version of the Trudinger-Moser inequality for functions in $Y_{\mathrm{rad}}$ which will plays an important role in our arguments. In \cite{Ad-Sand}, the authors have obtained a singular version of the Trudinger-Moser inequality. More precisely, they proved that if $\Omega$ is a bounded domain in $\mathbb{R}^2$ containing the origin, $u\in H_0^1(\Omega)$ and $\beta\in[0,2)$, then there exists a positive constant $C=C(\alpha,\beta,\Omega)$ such that
\begin{equation}\label{TM-A&S}
\sup_{\|\nabla u\|_{L^2(\Omega)}\leq1}\int_{\Omega}|x|^{-\beta}e^{\alpha u^2}\,\mathrm{d}x\leq C
\end{equation}
if and only if $0<\alpha\leq4\pi(1-\beta/2)$. On the other hand, in \cite[Proposition 1.1]{Calanchi}, the authors  have proved that if $u\in H_{0,\textrm{rad}}^1(B_1)$ and $\beta>0$, then there exists a positive constant $C=C(p,\beta)$ such that
\begin{equation}\label{TM Calanchi}
\sup_{\|\nabla u\|_{L^2(B_1)}\leq1}\int_{B_1}|x|^{\beta}(e^{p u^2}-1-pu^2)\,\mathrm{d}x\leq C
\end{equation}
if and only if $0<p\leq4\pi(1+\beta/2)$.
\begin{Proposition}\label{TM}
Suppose that $(V,Q)\in \mathcal{K}$. Then, for any $u\in Y_{\mathrm{rad}}$ and
$\alpha>0$, we have that $Q(|x|)^{\frac{4}{4-\mu}}(e^{\alpha u^2}-1)\in L^1(\mathbb{R}^2)$. Furthermore, if
$0<\alpha< 4\pi(1+\frac{2b_0}{4-\mu})$ then
$$
\sup_{u\in Y_{\mathrm{rad}},\ \|u\|\leq1}\int_{\mathbb{R}^2}Q(|x|)^{\frac{4}{4-\mu}}(e^{\alpha u^2}-1)\,\mathrm{d}x<\infty. 
$$
\end{Proposition}

\begin{proof}
Let $R_1>0$ such that Lemma \ref{lema-radial} be valid for $|x|=R_1$. Fixing $R>R_1$, we recall that by the hypothesis $(Q0)$ there exists $C_1>0$ such that
\begin{equation}\label{H-K}
Q(|x|)\leq C_1|x|^b,\ \text{for all}\ |x|\geq R,\ \text{and}\ \ Q(|x|)\leq C_1|x|^{b_0},\ \text{for all}\ |x|\leq R
\end{equation}
and we set
$$
I_1(\alpha,u)=\int_{B_R}Q(|x|)^{\frac{4}{4-\mu}}\left(e^{\alpha u^2} -1\right)\mathrm{d}x,\quad I_2(\alpha,u)=\int_{B_R^c}Q(|x|)^{\frac{4}{4-\mu}}\left(e^{\alpha u^2} -1\right)\mathrm{d}x.
$$
Thus, $$ \int_{\mathbb{R}^2}Q(|x|)^{\frac{4}{4-\mu}} \left(e^{\alpha u^2} -1\right)\mathrm{d}x = I_1( \alpha,u) +I_2( \alpha,u). $$ 
Now, we are going to estimate $ I_1(\alpha,u) $ and $ I_2(\alpha,u). $ First, using \eqref{H-K} and Corollary \ref{tec3} (or \eqref{est1}) with $q=\frac{4-\mu}{2}$, for $ u \in Y_{\mathrm{rad}}$, one has 
\begin{align*}
I_2(\alpha,u)&=\int_{B_R^c}Q(|x|)^{\frac{4}{4-\mu}}\sum_{j=1}^{\infty}\frac{\alpha^ju^{2j}}{j!}\,\mathrm{d}x\\
&=\int_{B_R^c}Q(|x|)^{\frac{4}{4-\mu}}\sum_{j=2}^{\infty}\frac{\alpha^ju^{2j}}{j!}\,\mathrm{d}x+\alpha\int_{B_R^c}Q(|x|)^{\frac{4}{4-\mu}}u^2\,\mathrm{d}x\\
&\leq C_2\int_{B_R^c}|x|^{\frac{4b}{4-\mu}}\sum_{j=2}^{\infty}\frac{\alpha^ju^{2j}}{j!}\,\mathrm{d}x+\alpha C(R)\|u\|^{2}\\
&=C_2\sum_{j=2}^{\infty}\frac{\alpha^j}{j!}\int_{B_R^c}|x|^{\frac{4b}{4-\mu}}u^{2j}\,\mathrm{d}x+\alpha C(R)\|u\|^{2}.
\end{align*}
On the other hand, by virtue of Lemma~\ref{lema-radial} and using the fact that $a>-2$ and $\frac{4b}{4-\mu}-j\frac{a+2}{2}+2\leq \frac{4b}{4-\mu}-a<0$, for all $j\geq 2,$ we get
\begin{align*}
\int_{B_R^c}|x|^{\frac{4b}{4-\mu}}u^{2j}\,\mathrm{d}x&\leq 2\pi\left(C\|u\|\right)^{2j}\int_R^{\infty}s^{\frac{4b}{4-\mu}-j\frac{a+2}{2}+1}\,\mathrm{d}s\\
& = 2\pi\left(C\|u\|\right)^{2j} \frac{R^{\frac{4b}{4-\mu}-j\frac{a+2}{2}+2}}{j \frac{a+2}{2} -\frac{4b}{4-\mu} -2}
\leq\frac{2\pi\left(C\|u\|\right)^{2j}R^{\frac{4b}{4-\mu}-a}}{a-\frac{4b}{4-\mu}},
\end{align*}
and consequently
\begin{align*}
I_2(\alpha,u)&\leq\frac{2\pi C_2R^{\frac{4b}{4-\mu}-a}}{a-\frac{4b}{4-\mu}}\sum_{j=2}^{\infty}\frac{\left(\alpha C^2\|u\|^2\right)^j}{j!}+\alpha C(R)\|u\|^{2}\\
&=C_3 e^{\alpha C^2\|u\|^2}+\alpha C(R)\|u\|^{2},
\end{align*}
where $C_3=C_3(a,b,R,\mu)$.
Hence, $ I_2( \alpha,u) < + \infty,\ \forall\ u \in Y_{\mathrm{rad}}. $ Moreover,
\begin{equation}\label{precise-estimate-outside}
\sup_{u\in Y_{\mathrm{rad}},\ \|u\|\leq1}I_2( \alpha, u) < +\infty,\ \forall\ \alpha > 0.
\end{equation}
Since $u\in H^{1}_{\mathrm{rad}}(B_R)$, we denote $u(x)=u(R)$ for $|x|=R$. Let $v:B_R\to \mathbb{R}$ defined by $v(x)=u(x)-u(R)$. With this we have that $v\in H^{1}_{0,\mathrm{rad}}(B_R)$ and $\|\nabla v\|_2=\|\nabla u\|_2$. Now, we can take $\varepsilon>0$ such that
$$
\alpha(1+\varepsilon) < 4\pi\left(1+\frac{2b_0}{4-\mu}\right).
$$
Using Young's inequality and Lemma \ref{lema-radial}, for $ x \in B_R, $ we have 
$$ 
u^2(x) \leq (1+\varepsilon) v^2(x) + (1+\frac{1}{\varepsilon}) u^2(R) \leq (1+\varepsilon) v^2(x) + (1+\frac{1}{\varepsilon}) C^{2}R^{-\frac{a+2}{2}}. 
$$ 
Furthermore,
we can choose $R>0$ large enough such that $ e^{(1+\varepsilon)C^{2} R^{- \frac{a+2}{2}}} \leq 1$. 
In order to estimate the integral $I_1(\alpha,u)$, we have two cases to analyze:\\

\noindent\textbf{Case~1:} $b_0>0$. In view of \eqref{H-K}, there
exists a positive constant $C_4$ depending on $ R $ such that
\begin{equation}\label{eq:dentro da bola1}
\begin{aligned}
I_1(\alpha,u)&\leq C_4\int_{B_R}|x|^{b_0\frac{4}{4-\mu}}(e^{\alpha u^2}-1)\,\mathrm{d}x\\
&\leq C_4\int_{B_R}|x|^{b_0\frac{4}{4-\mu}}[e^{\alpha(1+\varepsilon)v^2}-1-\alpha(1+\varepsilon) v^2] \mathrm{d}x +C_4\alpha(1+\varepsilon)\int_{B_R}|x|^{b_0\frac{4}{4-\mu}}v^2 \mathrm{d}x\\
&\leq C_4R^{b_0\frac{4}{4-\mu}+2}\int_{B_1}|x|^{b_0\frac{4}{4-\mu}}\left[e^{\alpha(1+\varepsilon)w^2}-1-\alpha(1+\varepsilon) w^2\right]\mathrm{d}x +C_5\alpha(1+\varepsilon)R^{b_0\frac{4}{4-\mu}},
\end{aligned}
\end{equation}
where $w(x)=v(Rx)$, $x\in B_1$, and we also have used a suitable change of variables, the Poincaré inequality and $\|w\|_{L^2(B_1)}=\|v\|_{L^2(B_R)}$. Hence, we can apply the Trudinger-Moser inequality \eqref{TM Calanchi} to deduce that 
$$ 
I_1( \alpha,u) < +\infty,\ \forall\ u \in Y_{\mathrm{rad}},
$$ 
and since $\alpha(1+\varepsilon) < 4\pi(1+2b_0/(4-\mu))$
$$
\sup_{u\in Y_{\mathrm{rad}},\ \|u\|\leq1}I_1( \alpha, u) < +\infty,\ \forall\ 0 < \alpha < 4\pi\left(1+\frac{2b_0}{4-\mu}\right).
$$
\textbf{Case~2:} $-\frac{4-\mu}{2}<b_0\leq0$. Again by \eqref{H-K} and the singular Trudinger-Moser inequality due to Adimurthi and K. Sandeep \eqref{TM-A&S} we have
$$
I_1(\alpha,u)\leq C_6\int_{B_R}\frac{e^{\alpha u^2} -1}{|x|^{-b_0\frac{4}{4-\mu}}}\,\mathrm{d}x\leq C_6\int_{B_R}\frac{e^{\alpha(1+\varepsilon) v^2}}{|x|^{-b_0\frac{4}{4-\mu}}}\,\mathrm{d}x<\infty,\ \forall\ u \in Y_{\mathrm{rad}},
$$
since $\frac{\alpha(1+\varepsilon)}{4\pi}+\frac{-b_0\frac{4}{4-\mu}}{2}< 1 \Leftrightarrow \alpha(1+\varepsilon)< 4\pi(1+\frac{2b_0}{4-\mu})$. Moreover,
\begin{equation}\label{precise-estimate-inside}
\sup_{u\in Y_{\mathrm{rad}},\ \|u\|\leq1}I_1( \alpha, u) < +\infty,\ \forall\ 0 < \alpha \leq 4\pi(1+\frac{2b_0}{4-\mu}).
\end{equation}
Thereby, from \eqref{precise-estimate-outside} and \eqref{precise-estimate-inside} we conclude that
$$
\sup_{u\in Y_{\mathrm{rad}},\ \|u\|\leq1}\int_{\mathbb{R}^2}Q(|x|)^{\frac{4}{4-\mu}}(e^{\alpha u^2}-1)\,\mathrm{d}x<\infty,
$$
and this ends the proof of the theorem.
\end{proof}


\section{The variational setting}
Using assumption $(f_1)$ we have $f(0)=0$ and we may assume, without loss of generality, that $f(s)=0$ for all $s\leq 0$. In what follows, we establish the necessary functional framework where solutions are naturally studied by variational methods. From $(f_1)$ and given $\varepsilon>0$, if $f(s)$ satisfies \eqref{def.cresc.critico}, then for $\alpha>\alpha_0$ and $p\geq1$ there exists $C_1=C_1(\varepsilon,p,\alpha)>0$ such that
\begin{equation}\label{cresc-f-critica}
f(s)\leq\varepsilon|s|^{\frac{2-\mu}{2}}+C_1|s|^{p-1}(e^{\alpha s^2}-1),\quad\forall s\in\mathbb{R}.
\end{equation}
Hence, there exists $C_2=C_2(\varepsilon,p,\alpha)>0$ satisfying
\begin{equation}\label{cresc-F}
F(s)\leq\frac{\varepsilon}{2}s^{\frac{4-\mu}{2}}+C_2|s|^p(e^{\alpha s^2}-1),\quad\forall s\in\mathbb{R}.
\end{equation}
The Euler-Lagrange functional associated to equation~\eqref{P} is given by
$$
J(u):=\dfrac{1}{2}\int_{\mathbb{R}^{2}}|\nabla u|^{2}\,\mathrm{d}x+\dfrac{1}{2}\int_{\mathbb{R}^{2}}V(|x|)u^{2}\,\mathrm{d}x-\dfrac{1}{2}\int_{\mathbb{R}^2}\left[|x|^{-\mu}*(Q(|x|)F(u))\right]Q(|x|)F(u)\,\mathrm{d}x.
$$
In order to control the nonlocal term $\int_{\mathbb{R}^2}\left[|x|^{-\mu}*(Q(|x|)F(u))\right]Q(|x|)F(u)\,\mathrm{d}x$, we need the well known Hardy-Littlewood-Sobolev inequality.
\begin{Lemma}[Hardy-Littlewood-Sobolev inequality]
Let $1<r,t<\infty$ and $0<\mu<N$ with $\frac{1}{t}+\frac{\mu}{N}+\frac{1}{r}=2$. If $g\in L^{t}(\mathbb{R}^{N})$ and $h\in L^{r}(\mathbb{R}^{N})$, then there exists a positive sharp constant $C(t,N,\mu,r)$, independent of $g$ and $h$, such that
$$
\int_{\mathbb{R}^{N}}\int_{\mathbb{R}^{N}}\frac{g(x)h(y)}{|x-y|^{\mu}}\,\mathrm{d}x\,\mathrm{d}y\leq C(t,N,\mu,r)|g|_t|h|_r.
$$
If $N=2$ and $t=r=\frac{4}{4-\mu}$, then
$$
C(t,N,\mu,r)=C(\mu)=\pi^{\frac{\mu}{2}}\frac{\Gamma\left(1-\frac{\mu}{2}\right)}{\Gamma\left(2-\frac{\mu}{2}\right)}\left[\frac{\Gamma(1)}{\Gamma(2)}\right]^{-1+\frac{\mu}{2}}=\frac{2\pi^{\frac{\mu}{2}}}{2-\mu},
$$
where $\Gamma(\cdot)$ is the usual Gamma function.

\begin{Remark}\label{obs}
 We can observe that the Hardy-Littlewood-Sobolev inequality still holds for $N = 2$. so, motivated by this fact, it seems to be interesting, at least from the mathematical point of view, to ask if the existence of nontrivial solution for \eqref{P} still holds for nonlinearities $f$ with exponential subcritical (or critical) growth in $\mathbb{R}^{2}$.
\end{Remark}

\end{Lemma}
Next, we shall verify that $J$ is well defined on the space $Y_{\mathrm{rad}}$. Notice that by the Hardy-Littlewood-Sobolev inequality with $t=r=\frac{4}{4-\mu}$, we have
\begin{align*}
\left|\int_{\mathbb{R}^2}\left[|x|^{-\mu}*(Q(|x|)F(u))\right]Q(|x|)F(u)\,\mathrm{d}x\right|&\leq C|Q(|x|)F(u)|^{2}_{\frac{4}{4-\mu}}\\
&=C\left(\int_{\mathbb{R}^2}Q(|x|)^{\frac{4}{4-\mu}}|F(u)|^{\frac{4}{4-\mu}}\,\mathrm{d}x\right)^{\frac{4-\mu}{2}}.
\end{align*}
By \eqref{cresc-F} we have
$$
\int_{\mathbb{R}^2}Q(|x|)^{\frac{4}{4-\mu}}|F(u)|^{\frac{4}{4-\mu}}\,\mathrm{d}x\leq C_3\int_{\mathbb{R}^2}Q(|x|)^{\frac{4}{4-\mu}}u^2\,\mathrm{d}x+C_4\int_{\mathbb{R}^2}Q(|x|)^{\frac{4}{4-\mu}}|u|^p(e^{\alpha u^2}-1)\,\mathrm{d}x.
$$
From Lemma \ref{tec1} with $q=\frac{4}{4-\mu}$, we infer that the integral $\int_{\mathbb{R}^2}Q(|x|)^{\frac{4}{4-\mu}}u^2\,\mathrm{d}x$ is finite. Hence, $Q(|x|)|u|^{\frac{4-\mu}{2}}\in L^{\frac{4}{4-\mu}}(\mathbb{R}^{2})$. The second integral can be estimated as follow. By Hölder's inequality, we have
$$
\int_{\mathbb{R}^2}Q(|x|)^{\frac{4}{4-\mu}}|u|^p(e^{\alpha u^2}-1)\,\mathrm{d}x
\leq\left(\int_{\mathbb{R}^2}Q(|x|)^{\frac{4}{4-\mu}}|u|^{2p}\,\mathrm{d}x\right)^{\frac{1}{2}}\left(\int_{\mathbb{R}^2}Q(|x|)^{\frac{4}{4-\mu}}(e^{2\alpha u^2}-1)\,\mathrm{d}x\right)^{\frac{1}{2}}<\infty,
$$
since Lemma \ref{tec1} and Proposition \ref{TM} hold. Consequently, $Q(|x|)F(u)\in L^{\frac{4}{4-\mu}}(\mathbb{R}^{2})$. Thus, $J$ is well defined on $Y_{\mathrm{rad}}$ and using standard arguments we can show that $J\in C^1(Y_{\mathrm{rad}},\mathbb{R})$ with derivative given by
$$
J'(u)\phi=\int_{\mathbb{R}^2}(\nabla u\cdot\nabla\phi+V(|x|)u\phi)\,\mathrm{d}x-\int_{\mathbb{R}^2}\left[|x|^{-\mu}*(Q(|x|)F(u))\right]Q(|x|)f(u)\phi\,\mathrm{d}x.
$$
In the following, we will need the next lemma.
\begin{Lemma}\label{tec4}
Let $\alpha>0$ and $0<\rho<\sqrt{\frac{2\pi}{\alpha}(1+\frac{2b_0}{4-\mu})}$. If $p\geq 1$ and $\|u\|\leq\rho$, then there exists $C>0$ such that 
$$
\int_{\mathbb{R}^2}Q(|x|)^{\frac{4}{4-\mu}}|u|^p(e^{\alpha u^2}-1)\,\mathrm{d}x\leq C\|u\|^{p}.
$$
\end{Lemma}

\begin{proof}
From Hölder's inequality, Corollary \ref{tec3} and Proposition \ref{TM}, we have
\begin{align*}
\int_{\mathbb{R}^2}Q(|x|)^{\frac{4}{4-\mu}}|u|^p(e^{\alpha u^2}-1)\,\mathrm{d}x
&\leq\left(\int_{\mathbb{R}^2}Q(|x|)^{\frac{4}{4-\mu}}|u|^{2p}\,\mathrm{d}x\right)^{\frac{1}{2}}\left(\int_{\mathbb{R}^2}Q(|x|)^{\frac{4}{4-\mu}}(e^{2\alpha u^2}-1)\,\mathrm{d}x\right)^{\frac{1}{2}}\\
&\leq C_1\|u\|^{p}\left(\int_{\mathbb{R}^2}Q(|x|)^{\frac{4}{4-\mu}}(e^{2\alpha\|u\|^{2}(u/\|u\|)^2}-1)\,\mathrm{d}x\right)^{\frac{1}{2}}\\
&\leq C\|u\|^{p},
\end{align*}
since that $2\alpha\|u\|^{2}\leq 2\alpha\rho^{2}<4\pi(1+\frac{2b_0}{4-\mu})$.
\end{proof}

Next, we shall justify that a critical point of the functional $J$ is exactly a weak solution of the problem \eqref{P}. For this we will need the following result:
\begin{Proposition}\label{tec5}
Assume that $(V,Q)\in \mathcal{K}$, \eqref{def.cresc.critico} and $(f_1)$ hold. Then, for each $u\in Y_{\mathrm{rad}}$, there exists $C=C(b_0,\|u\|)>0$ such that
$$
\left|\int_{\mathbb{R}^2}\left[|x|^{-\mu}*(Q(|x|)F(u))\right]Q(|x|)f(u)w\,\mathrm{d}x\right|\leq C\|w\|,\ \text{for all}\ w\in Y.
$$
\end{Proposition}

\begin{proof}
Let $w\in Y$. By the Hardy-Littlewood-Sobolev inequality with $t=r=\frac{4}{4-\mu}$, we have
$$
\left|\int_{\mathbb{R}^2}\left[|x|^{-\mu}*(Q(|x|)F(u))\right]Q(|x|)f(u)w\,\mathrm{d}x\right|\leq C_1|Q(|x|)F(u)|_{\frac{4}{4-\mu}}|Q(|x|)f(u)w|_{\frac{4}{4-\mu}}.
$$
Now, let us to estimate the term $|Q(|x|)f(u)w|_{\frac{4}{4-\mu}}$. From \eqref{cresc-f-critica} with $p=1$, we get
\begin{align*}
\int_{\mathbb{R}^2}&Q(|x|)^{\frac{4}{4-\mu}}|f(u)w|^{\frac{4}{4-\mu}}\,\mathrm{d}x\\
&\leq C_2\underbrace{\int_{\mathbb{R}^2}Q(|x|)^{\frac{4}{4-\mu}}u^{\frac{2(2-\mu)}{4-\mu}}|w|^{\frac{4}{4-\mu}}\,\mathrm{d}x}_{I_1}+C_3\underbrace{\int_{\mathbb{R}^2}Q(|x|)^{\frac{4}{4-\mu}}(e^{\alpha u^2}-1)^{\frac{4}{4-\mu}}|w|^{\frac{4}{4-\mu}}\,\mathrm{d}x}_{I_2}.
\end{align*}
Let us to analyze the integrals $I_1$ and $I_2$. Fixing $R>0$, we have
$$
I_1=\underbrace{\int_{B_R}Q(|x|)^{\frac{4}{4-\mu}}u^{\frac{2(2-\mu)}{4-\mu}}|w|^{\frac{4}{4-\mu}}\,\mathrm{d}x}_{I_1^{1}}+\underbrace{\int_{B_R^{c}}Q(|x|)^{\frac{4}{4-\mu}}u^{\frac{2(2-\mu)}{4-\mu}}|w|^{\frac{4}{4-\mu}}\,\mathrm{d}x}_{I_1^{2}}.
$$
Using the assumptions $(V0),\ (Q0)$ and  Hölder's inequality, one has
\begin{align}\label{est-I12}
I_1^{2}&\leq C \int_{B_R^{c}}|x|^{\frac{4b}{4-\mu}}u^{\frac{2(2-\mu)}{4-\mu}}|w|^{\frac{4}{4-\mu}}\,\mathrm{d}x\leq C \int_{B_R^{c}}|x|^a u^{\frac{2(2-\mu)}{4-\mu}}|w|^{\frac{4}{4-\mu}}\,\mathrm{d}x\nonumber\\
&\leq C_4 \int_{B_R^{c}}V(|x|) u^{\frac{2(2-\mu)}{4-\mu}}|w|^{\frac{4}{4-\mu}}\,\mathrm{d}x\nonumber\\
&\leq C_4 \left(\int_{B_R^{c}}V(|x|) u^2\,\mathrm{d}x\right)^{\frac{2-\mu}{4-\mu}} \left(\int_{B_R^{c}}V(|x|) w^2\,\mathrm{d}x\right)^{\frac{2}{4-\mu}}\nonumber\\
&\leq C_4 \left(\int_{B_R^{c}}V(|x|) u^2\,\mathrm{d}x\right)^{\frac{2-\mu}{4-\mu}}\|w\|^{\frac{4}{4-\mu}}.
\end{align}
Now, let $p>1$ sufficiently close to $1$ such that $pb_0>-\frac{4-\mu}{2}$. Thus, using again Hölder's inequality, Corollary \ref{tec3} with $q=\frac{4-\mu}{2}$ and the hypothesis $(Q0)$, we get the following estimate to $I_1^{1}$.
\begin{align*}
I_1^{1}&\leq \left(\int_{B_R}Q(|x|)^{\frac{4}{4-\mu}} u^2\,\mathrm{d}x\right)^{\frac{2-\mu}{4-\mu}} \left(\int_{B_R}Q(|x|)^{\frac{4}{4-\mu}} w^2\,\mathrm{d}x\right)^{\frac{2}{4-\mu}}\\
&\leq C(\|u\|)\left[ \left(\int_{B_R}|x|^{\frac{4b_0p}{4-\mu}}\,\mathrm{d}x\right)^{\frac{1}{p}}\left(\int_{B_R}|w|^{\frac{2p}{p-1}}\,\mathrm{d}x\right)^{\frac{p-1}{p}}\right]^{\frac{2}{4-\mu}}\\
&\leq C(\|u\|,R)\left[\left(\int_{B_{2R}}|\varphi w|^{\frac{2p}{p-1}}\,\mathrm{d}x\right)^{\frac{p-1}{p}}\right]^{\frac{2}{4-\mu}},
\end{align*}
where $\varphi$ is a radial cut-off function in $C_{0}^{\infty}(\mathbb{R}^{2})$ verifying
$$
0\leq \varphi\leq 1\ \text{in}\ \mathbb{R}^{2},\ \varphi\equiv 1\ \text{in}\ B_{R},\ \varphi\equiv 0\ \text{in}\ B_{2R}^{c}\ \ \text{and}\ \ |\nabla\varphi|\leq C\ \text{in}\ \mathbb{R}^{2}.
$$ 
Hence, $\varphi w\in H_0^{1}(B_{2R})$. Denoting by $m_R:=\min_{x\in B_{2R}\setminus B_R}V(|x|)>0$, it follows from the last inequality and the embedding $H_0^{1}(B_{2R})\hookrightarrow L^{\frac{2p}{p-1}}(B_{2R})$ that
\begin{align}\label{est-I11}
I_1^{1}&\leq C_5\left(\int_{B_{2R}}|\nabla(\varphi w)|^2\,\mathrm{d}x\right)^{\frac{2}{4-\mu}}\leq C_5\left(\int_{B_{2R}\setminus B_R}|w\nabla\varphi|^2\,\mathrm{d}x+C_6\int_{B_{2R}}|\nabla w|^2\,\mathrm{d}x\right)^{\frac{2}{4-\mu}}\nonumber\\
&\leq C_5\left(\frac{1}{m_R}\int_{B_{2R}\setminus B_R}V(|x|)w^2\,\mathrm{d}x+C_6\int_{\mathbb{R}^{2}}|\nabla w|^2\,\mathrm{d}x\right)^{\frac{2}{4-\mu}}\leq C_7\|w\|^{\frac{4}{4-\mu}}.
\end{align}
Thereby, from \eqref{est-I11} and \eqref{est-I12}, we get the following estimate for the integral $I_1$:
\begin{equation}\label{est-I1}
I_1=I_1^{1}+I_1^{2}\leq C(\|u\|,R)\|w\|^{\frac{4}{4-\mu}}.
\end{equation}
Next, we shall analyze $I_2$. Fixed again $R>0$, by Hölder's inequality and Proposition \ref{TM} we get
\begin{align*}
I_2&\leq \int_{B_R}Q(|x|)^{\frac{4}{4-\mu}}(e^{\frac{4\alpha}{4-\mu} u^2}-1)|w|^{\frac{4}{4-\mu}}\,\mathrm{d}x+\int_{B_R^{c}}Q(|x|)^{\frac{4}{4-\mu}}(e^{\frac{4\alpha}{4-\mu} u^2}-1)|w|^{\frac{4}{4-\mu}}\,\mathrm{d}x\\
&\leq \left(\int_{B_R}Q(|x|)^{\frac{4}{4-\mu}}(e^{\frac{4\alpha}{2-\mu} u^2}-1)\,\mathrm{d}x\right)^{\frac{2-\mu}{4-\mu}}\left(\int_{B_R}Q(|x|)^{\frac{4}{4-\mu}}w^2\,\mathrm{d}x\right)^{\frac{2}{4-\mu}}\\
&\ \ \ \ +\left(\int_{B_R^{c}}Q(|x|)^{\frac{4}{4-\mu}}(e^{\frac{4\alpha}{2-\mu} u^2}-1)\,\mathrm{d}x\right)^{\frac{2-\mu}{4-\mu}}\left(\int_{B_R^{c}}Q(|x|)^{\frac{4}{4-\mu}}w^2\,\mathrm{d}x\right)^{\frac{2}{4-\mu}}\\
&\leq C_8\left[\left(\int_{B_R}Q(|x|)^{\frac{4}{4-\mu}}w^2\,\mathrm{d}x\right)^{\frac{2}{4-\mu}}+\left(\int_{B^{c}_R}Q(|x|)^{\frac{4}{4-\mu}}w^2\,\mathrm{d}x\right)^{\frac{2}{4-\mu}}\right].
\end{align*}
Similar to the estimate of $I_1^{1}$, there exists $C_9>0$ such that
$$
\left(\int_{B_R}Q(|x|)^{\frac{4}{4-\mu}}w^2\,\mathrm{d}x\right)^{\frac{2}{4-\mu}}\leq C_9\|w\|^{\frac{4}{4-\mu}}.
$$
On the other hand, by assumptions $(V0)$ and $(Q0)$ we reach
\begin{align*}
\int_{B^{c}_R}Q(|x|)^{\frac{4}{4-\mu}}w^2\,\mathrm{d}x\leq C\int_{B^{c}_R}|x|^{\frac{4b}{4-\mu}}w^2\,\mathrm{d}x\leq C\int_{B^{c}_R}|x|^a w^2\,\mathrm{d}x\leq C\int_{B^{c}_R}V(|x|)w^2\,\mathrm{d}x.
\end{align*}
Consequently, we obtain the following estimate to $I_2$:
\begin{equation}\label{est-I2}
I_2\leq C_9\|w\|^{\frac{4}{4-\mu}}+C\left(\int_{B^{c}_R}V(|x|)w^2\,\mathrm{d}x\right)^{\frac{2}{4-\mu}}\leq C\|w\|^{\frac{4}{4-\mu}}+C\|w\|^{\frac{4}{4-\mu}}= C\|w\|^{\frac{4}{4-\mu}}.
\end{equation}
Therefore, from \eqref{est-I1} and \eqref{est-I2} we deduce that
$$
\int_{\mathbb{R}^2}Q(|x|)^{\frac{4}{4-\mu}}|f(u)w|^{\frac{4}{4-\mu}}\,\mathrm{d}x\leq C_2I_1+C_3I_2\leq C_1(\|u\|,R)\|w\|^{\frac{4}{4-\mu}}+C_2(\|u\|,R)\|w\|^{\frac{4}{4-\mu}},
$$
and so
$$
\left(\int_{\mathbb{R}^2}Q(|x|)^{\frac{4}{4-\mu}}|f(u)w|^{\frac{4}{4-\mu}}\,\mathrm{d}x\right)^{\frac{4-\mu}{4}}\leq C(\|u\|,R)\|w\|,
$$
and the proof of the proposition is complete.
\end{proof}

The next proposition shows that $Y_{\mathrm{rad}}$ actually is, in some sense, a natural constraint for finding weak solutions of problem \eqref{P}.
\begin{Proposition}
If $u\in Y_{\mathrm{rad}}$ is a critical point of $J$, then $u$ is a weak solution of problem \eqref{P}, that is, it holds the equality \eqref{weak-sol-def}.
\end{Proposition}

\begin{proof}
Firstly, due to Proposition \ref{tec5}, the linear functional $T_u:Y\rightarrow\mathbb R$ defined by
$$
T_u (w) := \int_{\mathbb{R}^2}\nabla u\cdot\nabla w\,\mathrm{d}x+\int_{\mathbb{R}^2}V(|x|)u w\,\mathrm{d}x-\int_{\mathbb{R}^2}\left[|x|^{-\mu}*(Q(|x|)F(u))\right]Q(|x|)f(u)w\,\mathrm{d}x
$$
is well defined and continuous on $Y$, since
$$
\left|T_u (w)\right|\leq \|u\|\|w\|+C\|w\|=(\|u\|+C)\|w\|,\ \text{for all}\ w\in Y.
$$
Now, taking into account that $u\in Y_{\mathrm{rad}}$ is a critical point of $J$, we infer that $T_u (w)=0,\ \text{for all}\ w\in Y_{\mathrm{rad}}$. So, by the Riesz Representation Theorem in the space $Y$ with the inner product \eqref{inner-def}, there exists a unique $\hat{u}\in Y$ such that $T_u(\hat{u})=\|\hat{u}\|^{2}=\|T_u\|_{Y'}$, where $Y'$ denotes the dual space of $Y$. Let $O(2)$ denote the group of orthogonal transformations in $\mathbb{R}^{2}$. Then, by using change of variables, one has for each $w\in Y$
$$
T_u(gw)=T_u(w)\quad\text{and}\quad \|gw\|=\|w\|,\ \ \text{for all}\ g\in O(2),
$$
whence, applying with $w=\hat{u}$, one deduces, by uniqueness, $g\hat{u}=\hat{u}$, for all $g\in O(2)$, which means that $\hat{u}\in Y_{\mathrm{rad}}$. Hence, since $T_u(w)=0$ for all $w \in Y_{\mathrm{rad}}$, one has $T_u(\hat{u})=0$, that is, $\|T_u\|_{Y'}=0\Leftrightarrow T_u (w)=0,\ \text{for all}\ w\in Y,$ and therefore \eqref{weak-sol-def} ensues. This concludes the proof of the proposition.
\end{proof}


\section{The mountain pass structure}

In a standard way, one can check in the next lemma that the functional $J$ has the geometric structure required by the Mountain-Pass Theorem.

\begin{Lemma}\label{geometria}
Assume that $(V,Q)\in \mathcal{K}$ and $f$ satisfies \eqref{def.cresc.critico}, $(f_1)$ and $(f_2)$. Then\\
$i)$ there exist $\rho,\gamma>0$ such that
    $$
J(u)\geq\gamma,\ \text{for all}\ \ \|u\|=\rho;
    $$
$ii)$ there exists $e_{\star}\in Y_{\mathrm{rad}}$, with $\|e_{\star}\|>\rho$, such that $J(e_{\star})<0$.
\end{Lemma}

\begin{proof}
$i)$ Again by the Hardy-Littlewood-Sobolev inequality with $t=r=\frac{4}{4-\mu}$ and \eqref{cresc-F}, we have
\begin{align*}
J(u)&\geq \frac{1}{2}\|u\|^2-C\left(\int_{\mathbb{R}^2}Q(|x|)^{\frac{4}{4-\mu}}|F(u)|^{\frac{4}{4-\mu}}\,\mathrm{d}x\right)^{\frac{4-\mu}{2}}\\
&\geq \frac{1}{2}\|u\|^2-C_1\varepsilon\left(\int_{\mathbb{R}^2}Q(|x|)^{\frac{4}{4-\mu}}u^2\,\mathrm{d}x\right)^{\frac{4-\mu}{2}}-C_2\left(\int_{\mathbb{R}^2}Q(|x|)^{\frac{4}{4-\mu}}|u|^p(e^{\alpha u^2}-1)\,\mathrm{d}x\right)^{\frac{4-\mu}{2}}.
\end{align*}
Now, using Corollary \ref{tec3} with $q=\frac{4-\mu}{2}$ and Lemma  \ref{tec4} with $p>\frac{4}{4-\mu}$, we get
\begin{align*}
J(u)&\geq \frac{1}{2}\|u\|^2-C_3\varepsilon\|u\|^{2}-C_4\|u\|^{\frac{p(4-\mu)}{2}}\\
&=\left(\frac{1}{2}-C_3\varepsilon\right)\|u\|^{2}-C_4\|u\|^{\frac{p(4-\mu)}{2}}\\
&=\left(\frac{1}{2}-C_3\varepsilon\right)\rho^{2}-C_4\rho^{\frac{p(4-\mu)}{2}},
\end{align*}
since $0<\rho<\sqrt{\frac{2\pi}{\alpha}(1+\frac{2b_0}{4-\mu})}$. 
Now, taking $0<\varepsilon<\frac{1}{2C_3}$ and $\rho>0$ small enough, we deduce
$$
\beta:=\left(\frac{1}{2}-C_3\varepsilon\right)\rho^{2}-C_4\rho^{\frac{p(4-\mu)}{2}}>0,
$$
because $\frac{p(4-\mu)}{2}>2$ and the item is proved.

In order to verify $ii)$, fix $u_0\in Y_{\mathrm{rad}}\setminus\{0\}$ with $u_0\geq 0$ in $\mathbb{R}^{2}$. For each $t>0$, we set
$$
\Phi(t):=I\left(\frac{tu_0}{\|u_0\|}\right),
$$
where
$$
I(u):=\frac{1}{2}\int_{\mathbb{R}^2}\left[|x|^{-\mu}*(Q(|x|)F(u))\right]Q(|x|)F(u)\,\mathrm{d}x.
$$
From $(f_2)$, it follows that
\begin{align*}
\Phi'(t)&=I'\left(\frac{tu_0}{\|u_0\|}\right)\frac{u_0}{\|u_0\|}\\
&=\frac{1}{t}\int_{\mathbb{R}^2}\left[|x|^{-\mu}*\left(Q(|x|)F\left(\frac{tu_0}{\|u_0\|}\right)\right)\right]Q(|x|)f\left(\frac{tu_0}{\|u_0\|}\right)\frac{tu_0}{\|u_0\|}\,\mathrm{d}x\\
&\geq\frac{2\theta}{t}\frac{1}{2}\int_{\mathbb{R}^2}\left[|x|^{-\mu}*\left(Q(|x|)F\left(\frac{tu_0}{\|u_0\|}\right)\right)\right]Q(|x|)F\left(\frac{tu_0}{\|u_0\|}\right)\mathrm{d}x\\
&=\frac{2\theta}{t}\Phi(t),
\end{align*}
which yields
\begin{equation}\label{bala1-2geometria}
\frac{\Phi'(t)}{\Phi(t)}\geq\frac{2\theta}{t},\ \text{for all}\ t>0.
\end{equation}
Integrating \eqref{bala1-2geometria} over $[1,t\|u_0\|]$, with $t>1/\|u_0\|$, we get
$$
I(tu_0)\geq C_5t^{2\theta},
$$
where $C_5=I\left(\frac{u_0}{\|u_0\|}\right)\|u_0\|^{2\theta}$. Hence
$$
J(tu_0)\leq\frac{t^2}{2}\|u_0\|^2-C_5t^{2\theta}.
$$
Since $\theta>1$, we conclude that $I(tu_0)\rightarrow-\infty$
as $t\rightarrow+\infty$. Therefore, setting $e_{\star}:=tu_0$ with $t$ large
enough, the proof of the lemma is finished.
\end{proof}

Now, we recall that $(u_n)\subset Y_{\mathrm{rad}}$ is a Palais-Smale sequence at level $c\in\mathbb{R}$ for the functional $J$, (PS)$_c$ shortly, if
$$
J(u_n)\rightarrow c\quad \textrm{and}\quad J'(u_n)\rightarrow0,\ \textrm{as}\ n\rightarrow+\infty.
$$
We say that $J$ satisfies the (PS)$_c$ compactness condition if any (PS)$_c$ sequence has a convergent subsequence. The next lemma is an important tool in our analysis.

\begin{Lemma}\label{limitacao-de-seq-ps}
Let $c\in\mathbb{R}$ and $(u_n)\subset Y_{\mathrm{rad}}$ be a (PS)$_c$ sequence for $J$. Then $(u_n)$ is bounded in $Y_{\mathrm{rad}}$ and 
\begin{equation}\label{est-norma-de-seq-ps}
\|u_n\|^{2}\leq \frac{2\theta}{\theta-1}c+o_n(1).
\end{equation}
\end{Lemma}
\begin{proof}
Let $(u_n)\subset Y_{\mathrm{rad}}$ be a (PS)$_c$ for the functional $J$. Thus,
$$
\frac{1}{2}\|u_n\|^2-\dfrac{1}{2}\int_{\mathbb{R}^2}\left[|x|^{-\mu}*(Q(|x|)F(u_n))\right]Q(|x|)F(u_n)\,\mathrm{d}x\rightarrow c
$$
and
$$
\|u_n\|^2-\int_{\mathbb{R}^2}\left[|x|^{-\mu}*(Q(|x|)F(u_n))\right]Q(|x|)f(u_n)u_n\,\mathrm{d}x\rightarrow 0,
$$
as $n\rightarrow+\infty$. From the above convergences, assumption $(f_2)$ and taking into account that $\theta>1$, we deduce that
\begin{align*}
&c+o_n(1)+\frac{1}{2\theta}\|J'(u_n)\|\|u_n\| \geq J(u_n)-\frac{1}{2\theta}J'(u_n)u_n\\
&=\left(\frac{1}{2}-\frac{1}{2\theta}\right)\|u_n\|^2+\frac{1}{2\theta}\int_{\mathbb{R}^2}\left[|x|^{-\mu}*(Q(|x|)F(u_n))\right]Q(|x|)[f(u_n)u_n-\theta F(u_n)]\,\mathrm{d}x\\
&\geq\left(\frac{1}{2}-\frac{1}{2\theta}\right)\|u_n\|^2,
\end{align*}
which implies that $(u_n)$ is bounded in $Y_{\mathrm{rad}}$. For the second part, by the last inequality, the boundedness of $(u_n)$ in $Y_{\mathrm{rad}}$ and since $J'(u_n)\rightarrow0\ \textrm{as}\ n\rightarrow+\infty$, it follows that
$$
\frac{\theta-1}{2\theta}\|u_n\|^{2}\leq c+o_n(1)+C\|J'(u_n)\|=c+o_n(1),
$$
which achieves the desired estimate and we conclude the proof.
\end{proof}

\section{Proof of Theorems \ref{main-thm} and \ref{ground-state}}

To prove Theorems \ref{main-thm} and \ref{ground-state}, we need to establish the following compactness result:
\begin{Proposition}\label{ps}
	The energy functional $J$ satisfies the (PS)$_c$ condition for all $c\in \mathbb{R}$ satisfying
	$$
	c<c_0:=\frac{4-\mu}{\alpha_0}\left(1+\frac{2b_0}{4-\mu}\right)\frac{\pi(\theta-1)}{2\theta}.$$
\end{Proposition}
\begin{proof}
	Let $c<c_0$ and $(u_n)\subset Y_{\mathrm{rad}}$ be a (PS)$_c$ sequence for $J$. From Lemma~\ref{limitacao-de-seq-ps}, $(u_n)$ is bounded in $Y_{\mathrm{rad}}$. Without loss of generality, we may assume that $u_n\rightharpoonup u$ weakly in $Y_{\mathrm{rad}}$. We will prove that, going if necessary to a subsequence, $u_n\rightarrow u$ strongly in $Y_{\mathrm{rad}}$. By the convexity of the functional $\Psi(u):=\frac{1}{2}\|u\|^{2}$, we obtain
	\begin{align*}
	\frac{1}{2}\|u\|^{2}-\frac{1}{2}\|u_n\|^{2}&\geq \Psi'(u_n)(u-u_n)\\
	&=\int_{\mathbb{R}^2}\nabla u_n\cdot\nabla(u-u_n)\,\mathrm{d}x+\int_{\mathbb{R}^2}V(|x|)u_n(u-u_n)\,\mathrm{d}x\\
	&=J'(u_n)(u-u_n)+\int_{\mathbb{R}^2}\left[|x|^{-\mu}*(Q(|x|)F(u_n))\right]Q(|x|)f(u_n)(u-u_n)\,\mathrm{d}x.
	\end{align*}
	Now, we claim that 
	\begin{equation}\label{tec6}
	\int_{\mathbb{R}^2}\left[|x|^{-\mu}*(Q(|x|)F(u_n))\right]Q(|x|)f(u_n)(u-u_n)\,\mathrm{d}x\rightarrow0,\ \textrm{as}\ n\rightarrow+\infty.
	\end{equation}
	Assuming \eqref{tec6} as true, we get $\|u_n\|^{2}\leq\|u\|^{2}+o_n(1)$
	and therefore
	\begin{equation}\label{limsup}
	\limsup_{n\rightarrow+\infty}\|u_n\|^{2}\leq \|u\|^{2}.
	\end{equation}
	On the other hand, by the weak convergence
	$$
	\|u\|^{2}\leq \liminf_{n\rightarrow+\infty}\|u_n\|^{2},
	$$
	which together with \eqref{limsup} yields that $\|u_n\|^{2}\rightarrow \|u\|^{2}$. So $u_n\rightarrow u$ in $Y_{\textrm{rad}}$ and this concludes the proof. 
	
	Now, let us to prove \eqref{tec6}. By the Hardy-Littlewood-Sobolev inequality with $t=r=\frac{4}{4-\mu}$, we have
	$$
	\left|\int_{\mathbb{R}^2}\left[|x|^{-\mu}*(Q(|x|)F(u_n))\right]Q(|x|)f(u_n)(u_n-u)\,\mathrm{d}x\right|\leq C_1|Q(|x|)F(u_n)|_{\frac{4}{4-\mu}}|Q(|x|)f(u_n)(u_n-u)|_{\frac{4}{4-\mu}}.
	$$
	Writing $A_n^{1}=|Q(|x|)F(u_n)|_{\frac{4}{4-\mu}}$ and $A_n^{2}=|Q(|x|)f(u_n)(u_n-u)|_{\frac{4}{4-\mu}}$, we shall prove that $A_n^{1}\leq C,\ \text{for all}\ n\in\mathbb{N}$ and for some positive constant $C$, and $A_n^{2}\rightarrow0\ \textrm{as}\ n\rightarrow+\infty$. From \eqref{cresc-F}, Hölder's inequality and Lemma \ref{limitacao-de-seq-ps}, we get
	\begin{align*}
	A_n^{1}{^{\frac{4}{4-\mu}}}&\leq C_2\int_{\mathbb{R}^2}Q(|x|)^{\frac{4}{4-\mu}}u_n^2\,\mathrm{d}x+C_3\int_{\mathbb{R}^2}Q(|x|)^{\frac{4}{4-\mu}}|u_n|^{\frac{4}{4-\mu}}(e^{\alpha u_n^2}-1)^{\frac{4}{4-\mu}}\,\mathrm{d}x\\
	&\leq C_4\|u_n\|^{2}+C_5\|u_n\|^{\frac{4q'}{4-\mu}}\left(\int_{\mathbb{R}^2}Q(|x|)^{\frac{4}{4-\mu}}\left(e^{\frac{4q\alpha}{4-\mu}\|u_n\|^{2}(u_n/\|u_n\|)^2}-1\right)\mathrm{d}x\right)^{\frac{1}{q}}\\
	&\leq C_6+C_7\left(\int_{\mathbb{R}^2}Q(|x|)^{\frac{4}{4-\mu}}\left(e^{\frac{4q\alpha}{4-\mu}\|u_n\|^{2}(u_n/\|u_n\|)^2}-1\right)\mathrm{d}x\right)^{\frac{1}{q}},
	\end{align*}
	where $q>1$ is close to $1$ and $q'=\frac{q}{q-1}$. From \eqref{est-norma-de-seq-ps} and since $c<\frac{4-\mu}{\alpha_0}\left(1+\frac{2b_0}{4-\mu}\right)\frac{\pi(\theta-1)}{2\theta}$, there exist $n_0\in\mathbb{N}$ and $\delta>0$ such that
	$$
	\|u_n\|^{2}\leq \frac{4-\mu}{\alpha_0}\left(1+\frac{2b_0}{4-\mu}\right)\pi-\delta,\ \ \ \forall\ n> n_0.
	$$
	For $q>1$ close to $1$ and $\alpha>\alpha_0$ close to $\alpha_0$ we will still have for some $\hat{\delta}>0$ that
	$$
	\frac{4q\alpha}{4-\mu}\|u_n\|^{2}\leq \left(1+\frac{2b_0}{4-\mu}\right)4\pi-\hat{\delta},\ \ \ \forall\ n> n_0.
	$$
	Hence, from our version of the Trudinger-Moser inequality in Proposition \ref{TM}, we obtain
	\begin{equation}\label{tec7}
	\int_{\mathbb{R}^2}Q(|x|)^{\frac{4}{4-\mu}}\left(e^{\frac{4q\alpha}{4-\mu}\|u_n\|^{2}(u_n/\|u_n\|)^2}-1\right)\mathrm{d}x\leq C_8,
	\end{equation}
	which implies that $A_n^{1}\leq C$. In what follows, we will show that $A_n^{2}\rightarrow0,\ \textrm{as}\ n\rightarrow+\infty$. From \eqref{cresc-f-critica}, Hölder's inequality, Lemma \ref{limitacao-de-seq-ps} and \eqref{tec7}, we have
	\begin{align*}
	A_n^{2}{^{\frac{4}{4-\mu}}}&\leq C_9\int_{\mathbb{R}^2}Q(|x|)^{\frac{4}{4-\mu}}|u_n|^{\frac{2(2-\mu)}{4-\mu}}|u_n-u|^{\frac{4}{4-\mu}}\,\mathrm{d}x+C_{10}\int_{\mathbb{R}^2}Q(|x|)^{\frac{4}{4-\mu}}|u_n-u|^{\frac{4}{4-\mu}}(e^{\alpha u_n^2}-1)^{\frac{4}{4-\mu}}\,\mathrm{d}x\\
	&\leq\left(\int_{\mathbb{R}^2}Q(|x|)^{\frac{4}{4-\mu}}u_n^2\,\mathrm{d}x\right)^{\frac{2-\mu}{4-\mu}}\left(\int_{\mathbb{R}^2}Q(|x|)^{\frac{4}{4-\mu}}|u_n-u|^2\,\mathrm{d}x\right)^{\frac{2}{4-\mu}}\\
	&\ \ \ \ +C_{10}\left(\int_{\mathbb{R}^2}Q(|x|)^{\frac{4}{4-\mu}}\left(e^{\frac{4q\alpha}{4-\mu}u_n^2}-1\right)\mathrm{d}x\right)^{\frac{1}{q}}\left(\int_{\mathbb{R}^2}Q(|x|)^{\frac{4}{4-\mu}}|u_n-u|^{\frac{4q'}{4-\mu}}\,\mathrm{d}x\right)^{\frac{1}{q'}}\\
	&\leq C_{11}\left(\int_{\mathbb{R}^2}Q(|x|)^{\frac{4}{4-\mu}}|u_n-u|^2\,\mathrm{d}x\right)^{\frac{2}{4-\mu}}+C_{12}\left(\int_{\mathbb{R}^2}Q(|x|)^{\frac{4}{4-\mu}}|u_n-u|^{\frac{4q'}{4-\mu}}\,\mathrm{d}x\right)^{\frac{1}{q'}}.
	\end{align*}
	From Corollary \ref{tec2},
	$$
	\int_{\mathbb{R}^2}Q(|x|)^{\frac{4}{4-\mu}}|u_n-u|^2\,\mathrm{d}x\rightarrow 0\ \ \text{and}\ \int_{\mathbb{R}^2}Q(|x|)^{\frac{4}{4-\mu}}|u_n-u|^{\frac{4q'}{4-\mu}}\,\mathrm{d}x\rightarrow 0,\ \textrm{as}\ n\rightarrow+\infty.
	$$
	Therefore, $A_n^{2}\rightarrow0\ \textrm{as}\ n\rightarrow+\infty$, and the proof is completed.
\end{proof}

We have checked in Lemma \ref{geometria} that the functional $J$ satisfies the mountain-pass geometry. Thus, the minimax level can be characterized by
$$
0<c^{\star}:=\inf_{\gamma\in\Gamma}\max_{t\in[0,1]}J(\gamma(t)),
$$
where the set of paths $\Gamma$ is defined as
$$
\Gamma=\left\{\gamma\in C([0,1],Y_{\mathrm{rad}}):\gamma(0)=0\,\,\, \textrm{and}\,\,\, J(\gamma(1))<0\right\}.
$$

Next, we obtain an estimate for the minimax level $c^{\star}$ which will be crucial in our arguments. 
\begin{Proposition}[Minimax estimate]\label{minimax estimate}
	Suppose that $(f_3)$ is satisfied with
	$$
	\xi\geq\max\left\{\xi_1, \left[\frac{\frac{\|Q\|_{L^1(B_{1/2})}^2}{2}(q-1)\left(\frac{\xi_1^2}{q}\right)^{q/(q-1)}}{\frac{4-\mu}{\alpha_0}\left(1+\frac{2b_0}{4-\mu}\right)\frac{\pi(\theta-1)}{2\theta}}\right]^{(q-1)/2}\right\},
	$$
	where
	$$
	\xi_1:=\frac{[\pi+\|V\|_{L^1(B_1)}]^{\frac{1}{2}}}{\|Q\|_{L^1(B_{1/2})}}.
	$$
	Then, $c^{\star}<c_0$.
\end{Proposition}
\begin{proof}
	First, we are going to consider a function $\varphi_0\in C_{0,\textrm{rad}}^{\infty}(\mathbb{R}^2)$ given by $\varphi_0(x)=1$ if $|x|\leq 1/2$, $\varphi_0(x)=0$ if $|x|\geq 1$, $0\leq\varphi_0(x)\leq 1$ for all $x\in \mathbb{R}^2$ and $|\nabla \varphi_0(x)|\leq 1$ for all $x\in \mathbb{R}^2$. By $(f_3)$ we infer that if $\xi\geq\xi_1$ then
	$$
	\begin{aligned}
	J(\varphi_0)&\leq\frac{1}{2} \int_{B_1}[|\nabla \varphi_0|^2+V(|x|)\varphi_0^2]\,\mathrm{d}x-\frac{1}{2}\int_{B_1}\left[|x|^{-\mu}*(Q(|x|)F(\varphi_0))\right]Q(|x|)F(\varphi_0)\,\mathrm{d}x\\
	&< \frac{\pi}{2}+\frac{1}{2}\|V\|_{L^1(B_1)}-\frac{\xi^2}{2} \int_{B_{1/2}}\left[|x|^{-\mu}*Q(|x|)\right]Q(|x|)\,\mathrm{d}x\\
	&\leq\frac{1}{2}[\pi+\|V\|_{L^1(B_1)}]-\frac{\xi_1^2}{2}\|Q\|_{L^1(B_{1/2})}^2=0.
	\end{aligned}
	$$
	In particular, 
	\begin{equation}\label{bebe1}
	\int_{B_2}[|\nabla \varphi_0|^2+V(|x|)\varphi_0^2]\,\mathrm{d}x<\xi_1^2\|Q\|_{L^1(B_{1/2})}^2.
	\end{equation}
	According to the definition of $\varphi_0$, \eqref{bebe1} and the hypothesis on $\xi$, a simple computation shows that
	$$
	\begin{aligned}
	c^{\star}\leq \max_{t\in [0,1]}I(t\varphi_0)&\leq \max_{t\in [0,1]}\left[\frac{t^2}{2}\left(\int_{B_1}[|\nabla \varphi_0|^2+V(|x|)\varphi_0^2]\,\mathrm{d}x\right)-\frac{\xi^2}{2}t^{2q}\|Q\|_{L^1(B_{1/2})}^2\right]\\
	&< \max_{t\in [0,1]}\left[\frac{\xi_1^2}{2}\|Q\|_{L^1(B_{1/2})}t^2-\frac{\xi^2}{2}t^{2q}\|Q\|_{L^1(B_{1/2})}^2\right]\\
	&\leq \frac{\|Q\|_{L^1(B_{1/2})}^2}{2}\max_{t\geq 0}\left[\xi_1^2t^2-\xi^2t^{2q}\right].
	\end{aligned}
	$$
	Calculating the above maximum, we obtain
	$$
	c^{\star}< \frac{\|Q\|_{L^1(B_{1/2})}^2}{2}\frac{1}{\xi^{2/(q-1)}}(q-1)\left(\frac{\xi_1^2}{q}\right)^{q/(q-1)}.
	$$
	Thus, if 
	$$
	\xi\geq\left[\frac{\frac{\|Q\|_{L^1(B_{1/2})}^2}{2}(q-1)\left(\frac{\xi_1^2}{q}\right)^{q/(q-1)}}{\frac{4-\mu}{\alpha_0}\left(1+\frac{2b_0}{4-\mu}\right)\frac{\pi(\theta-1)}{2\theta}}\right]^{(q-1)/2}
	$$
	then $c^{\star}<c_{0}$ and proof of the proposition is done.
\end{proof}

\begin{proof}[Proof of Theorem \ref{main-thm}]
	By using Mountain-Pass Theorem without the Palais-Smale condition (see \cite{Ambro-Rabi}), there exists a (PS)$_{c^{\star}}$ sequence $(u_n)\subset Y_{\mathrm{rad}}$ for $J$. From Lemma \ref{limitacao-de-seq-ps}, $(u_n)$ is bounded in $Y_{\mathrm{rad}}$, then, up to  a subsequence, there exists $ u_\star \in Y_{\mathrm{rad}} $ such that $ u_n \rightharpoonup u_\star $ weakly in $ Y_{\mathrm{rad}}. $ Since, by Proposition \ref{ps}, the energy functional $J$ satisfies the (PS)$_{c^{\star}}$ condition, then $u_n\rightarrow u_\star$ strongly in $Y_{\mathrm{rad}}$. Therefore, $ J'(u_\star) = 0 $ and $ J(u_\star) = c^{\star} > 0. $ This ends the proof of Theorem \ref{main-thm}. 
\end{proof}

Now, we proceed to prove Theorem \ref{ground-state}.
\begin{proof}[Proof of Theorem \ref{ground-state}]
	By a ground state solution of equation \eqref{P} we mean a nontrivial solution $\tilde{u}\in Y_{\mathrm{rad}}$ of \eqref{P} such that 
	$$
	J(\tilde{u})=\min\{J(u):u\neq 0,\ u\in Y_{\mathrm{rad}}\ \text{is a critical point of}\ J\}.
	$$
	So, let
	$$
	M^{\star}:=\min_{u\in \mathcal{N}}J(u),
	$$
	where $\mathcal{N}$ is the Nehari manifold 
	$$
	\mathcal{N}:=\{J(u):u\in Y_{\mathrm{rad}}\setminus\{0\}\ \text{is a critical point of}\ J\}.
	$$
	For this aim, it is sufficient to prove that $c^{\star} \leq M^{\star}$. The Nehari manifold $\mathcal{N}$ is closely linked to the behavior of the function $h_u:t\rightarrow J(tu)$ for $t>0$. Such map is known as \textit{fibering map} that dates back to the fundamental works \cite{DP,Poho}. If $u\in \mathcal{N}$ then
	\begin{align*}
	h_u(t)&=\frac{t^{2}}{2}\|u\|^{2}-\frac{1}{2}\int_{\mathbb{R}^2}\left[|x|^{-\mu}*(Q(|x|)F(tu))\right]Q(|x|)F(tu)\,\mathrm{d}x,\\
	h_u'(t)&=t\|u\|^{2}-\int_{\mathbb{R}^2}\left[|x|^{-\mu}*(Q(|x|)F(tu))\right]Q(|x|)f(tu)u\,\mathrm{d}x.
	\end{align*}
	Since $J'(u)u=0$, as a direct consequence, we deduce
	$$
	h_u'(t)=t\int_{\mathbb{R}^2}\left[|x|^{-\mu}*Q(|x|)\left(\frac{F(u)f(u)}{u}-\frac{F(tu)f(tu)}{tu}\right)\right]Q(|x|)u^{2}\,\mathrm{d}x,
	$$
	for $t>0$. Taking into account that $f(s)/s$ is increasing for $s>0$ and $f(s)\geq 0$ for all $s\in\mathbb{R}$, we infer that $\frac{F(s)f(s)}{s}$ is also increasing for $s>0$. Hence, $h_u'(t)>0$ for $t\in(0,1)$ and $h_u'(t)<0$ for $t\in(1,\infty)$. Hence, observing $h_u'(1)=0$, we conclude that $J(u)=\max_{t\geq 0} J(tu)$. Setting $\gamma(t):=(t_0u)t$, for $t\in[0,1]$, where $t_0>1$ is sufficiently large such that $J(t_0u)<0$, we have $\gamma\in \Gamma$ ($\Gamma$ was given in the definition of $c^{\star}$), and so
	$$
	c^{\star}\leq \max_{t\in[0,1]}J(\gamma(t))\leq \max_{t\geq 0}J(tu)=J(u).
	$$
	Therefore, since $u\in \mathcal{N}$ is arbitrary, $c^{\star}\leq M^{\star}$.
\end{proof}

\section{Proof of Theorem \ref{main-thm2}}

In this section, we present the proof of our second existence result. We recall that, in this case, condition $(f_3)$ is replaced by conditions $(f_4)$ and $(f_5)$. In this context, we need to prove a new estimate for the minimax level $c^{\star}$.  For this aim, we consider the Moser sequence
$$
\widetilde{\omega}_n(x)=\frac{1}{\sqrt{2\pi}}\left\{
  \begin{aligned}
    &\sqrt{\log n}, & & \hbox{$0\leq |x|\leq \frac{1}{n}$,} \\
    &\frac{\log\frac{1}{|x|}}{\sqrt{\log n}}, & & \hbox{$\frac{1}{n}<|x|\leq 1$,} \\
    &0, & & \hbox{$|x|>1$}.
  \end{aligned}
\right.
$$
Next, we prove some useful properties of the Moser's sequence $(\widetilde{\omega}_n)\subset Y_{\textrm{rad}}$, which will be important in the sequel.
\begin{Lemma}
Assuming condition $(V0)$,  it holds $\|\widetilde{\omega}_n\|^2\leq 1+\delta_n$, where 
$$
\delta_n>0,\quad \delta_n \log n\rightarrow \frac{2M_1}{(a_0+2)^2}\quad and \quad 0<M_1=\sup_{0<|x|\leq 1}V(|x|)/|x|^{a_0}<\infty.
$$
\end{Lemma}
\begin{proof}
First, since $|\nabla[\log(1/|x|)]|^2=1/|x|^2$ and by using polar coordinates we have
\begin{equation}\label{soledade}
\int_{\mathbb{R}^2}\left|\nabla \widetilde{\omega}_n\right|^{2}\mathrm{d}x=\int_{B_1}\left|\nabla \widetilde{\omega}_n\right|^{2}\mathrm{d}x=\frac{1}{\log n}\displaystyle\int_{1/n}^{1}\frac{1}{r}\,\mathrm{d}r=1.
\end{equation}
Now, let us to estimate the integral $I_n:=\int_{\mathbb{R}^2}V(|x|) \widetilde{\omega}_n^{2}\,\mathrm{d}x$. According to condition $(V0)$, we have $0<M_1<\infty$. Thus, 
$$
\begin{aligned}
I_n\leq M_1\displaystyle\int_{B_1}|x|^{a_0} \widetilde{\omega}_n^{2}\,\mathrm{d}x&\leq M_1\displaystyle\log n\int_{0}^{1/n}r^{a_0+1}\mathrm{d}r+\frac{M_1}{\log n}\int_{1/n}^{1}\log^2\left(\frac{1}{r}\right)r^{a_0+1}\,\mathrm{d}r\\
&=M_1\frac{\log n}{n^{a_0+2}}\frac{1}{a_0+2}+\frac{M_1}{\log n}\int_{1/n}^{1}\log^2\left(\frac{1}{r}\right)r^{a_0+1}\,\mathrm{d}r=:\delta_n.
\end{aligned}
$$
Making the change of variable $t=\log\left(\frac{1}{r}\right)$ and integrating by parts twice, we get
\begin{align*}
\int_{1/n}^{1}\log^2\left(\frac{1}{r}\right)r^{a_0+1}\,\mathrm{d}r&=\displaystyle\int_{0}^{\log n}t^{2}e^{-(a_0+2)t}\,\mathrm{d}t\\
&=\frac{2}{(a_0+2)^3}-\frac{2}{(a_0+2)^3}\frac{1}{n^{a_0+2}}-\frac{2}{(a_0+2)^3}\frac{\log n}{n^{a_0+2}}-\frac{1}{a_0+2}\frac{\log^2n}{n^{a_0+2}}
\end{align*}
and therefore
$$
0<I_n\leq \delta_n=\frac{2M_1}{(a_0+2)^3}\frac{1}{\log n}-\frac{2M_1}{(a_0+2)^3}\frac{1}{n^{a_0+2}\log n}-\frac{2M_1}{(a_0+2)^3}\frac{1}{n^{a_0+2}}
$$
and
$$
\delta_n\log n\rightarrow \frac{2M_1}{(a_0+2)^3}.
$$
Moreover,
$$
\|\widetilde{\omega}_n\|^{2}=\displaystyle\int_{B_1}\left|\nabla \widetilde{\omega}_n\right|^{2}\mathrm{d}x+I_n\leq 1+\delta_n
$$
and this completes the proof.
\end{proof} 

\begin{Proposition}
If we also assume that $ \displaystyle\liminf_{r \to 0^+} Q(r)/r^{b_0} > 0$, then mountain pass level $c^{\star}$ satisfies the estimate
\begin{equation}\label{est-nivel}
c^{\star}<\frac{4-\mu}{\alpha_0}\left(1+\frac{2b_0}{4-\mu}\right)\frac{\pi}{2}.
\end{equation}
\end{Proposition}
\begin{proof}
If we define \begin{equation}\label{norm-moser-seq}
\omega_n=\widetilde{\omega}_n/\sqrt{1+\delta_n},
\end{equation} 
then, by Lemma \ref{soledade}, $\|\widetilde{\omega}_n\|\leq 1$. To prove \eqref{est-nivel}, it suffices to show that there exists some $n_0\in\mathbb{N}$ such that
$$
\max_{t\geq0}\left\{\frac{t^{2}}{2}-I(t\omega_{n_0})\right\}<\frac{4-\mu}{\alpha_0}\left(1+\frac{2b_0}{4-\mu}\right)\frac{\pi}{2},
$$
where the functional $I$ was defined in Lemma \ref{geometria} and $w_{n_0}$ is the normalized Moser's sequence term given by \eqref{norm-moser-seq}. In fact, taking $e_\star=t_0\omega_{n_0}$, with $t_0>0$ sufficiently large such that $J(e_\star)<0$, and defining $\gamma_0(t):=(t_0\omega_{n_0})t$, then we have $\gamma_0\in \Gamma$ and
\begin{align*}
c^{\star}=\inf_{\gamma\in\Gamma}\max_{t\in[0,1]}J(\gamma(t))\leq \max_{t\in[0,1]}J(\gamma_0(t))\leq \max_{t\geq 0}J(t\omega_{n_0})&=\max_{t\geq0}\left\{\frac{t^{2}}{2}-I(t\omega_{n_0})\right\}\\
&<\frac{4-\mu}{\alpha_0}\left(1+\frac{2b_0}{4-\mu}\right)\frac{\pi}{2}.
\end{align*}
Arguing by contradiction, suppose that for each $n\in\mathbb{N}$ there is  $t_n>0$ such that
$$
\frac{t_n^2}{2}-\frac{1}{2}\int_{\mathbb{R}^2}\left[|x|^{-\mu}*(Q(|x|)F(t_n\omega_n))\right]Q(|x|)F(t_n\omega_n)\,\mathrm{d}x\geq\frac{4-\mu}{\alpha_0}\left(1+\frac{2b_0}{4-\mu}\right)\frac{\pi}{2}.
$$
Since $F(s)\geq0$ for all $s\in\mathbb{R}$, we have
\begin{equation}\label{bala1-nivel}
t_n^2\geq\frac{4-\mu}{\alpha_0}\left(1+\frac{2b_0}{4-\mu}\right)\frac{\pi}{2}.
\end{equation}
One can easily check that at $t=t_n$,
$$
\frac{\mathrm{d}}{\mathrm{d}t}\left(\frac{t^2}{2}-\frac{1}{2}\int_{\mathbb{R}^2}\left[|x|^{-\mu}*(Q(|x|)F(t\omega_n))\right]Q(|x|)F(t\omega_n)\,\mathrm{d}x\right)=0,
$$
or equivalently
\begin{equation}\label{bala2-nivel}
t_n^2=\int_{\mathbb{R}^2}\left[|x|^{-\mu}*(Q(|x|)F(t_n\omega_n))\right]Q(|x|)f(t_n\omega_n)t_n\omega_n\,\mathrm{d}x.
\end{equation}
From $(f_3)$ and $(f_4)$, for all $0<\varepsilon<\beta_0$, there exists $R=R(\varepsilon)>0$ such that
\begin{equation}\label{bala3-nivel}
sf(s)F(s)\geq M_0^{-1}(\beta_0-\varepsilon)s^{\vartheta+1}e^{2\alpha_0s^2},\quad\forall\ s\geq R.
\end{equation}
By the condition added concerning the behavior of $ Q $ near $ 0, $ there exists a positive constant $ C_2 $ such that 
\begin{equation}\label{behQ-origin}
Q(\left|x\right|) \geq C_2 \left|x\right|^{b_0},\ \forall\ 0 < \left|x\right| < r_0. 
\end{equation}
Thus, for $n$ large such that $B_{1/n}\subset B_{r_0}$, we obtain
\begin{align*}
t_n^2&\geq \int_{B_{1/n}}\int_{B_{1/n}}\frac{Q(|x|)F(t_n\omega_n)Q(|y|)f(t_n\omega_n)t_n\omega_n}{|x-y|^{\mu}}\,\mathrm{d}x\,\mathrm{d}y\nonumber\\
&\geq M_0^{-1}(\beta_0-\varepsilon)t_n^{\vartheta+1}\omega_n^{\vartheta+1}e^{2\alpha_0t_n^2\omega_n^2}\int_{B_{1/n}}Q(|x|)\,\mathrm{d}x\int_{B_{1/n}}\frac{Q(|y|)}{|x-y|^{\mu}}\,\mathrm{d}y\nonumber\\
&\geq M_0^{-1}(\beta_0-\varepsilon)t_n^{\vartheta+1}\omega_n^{\vartheta+1}e^{2\alpha_0t_n^2\omega_n^2}\left(\frac{n}{2}\right)^{\mu}\int_{B_{1/n}}\int_{B_{1/n}}Q(|x|)Q(|y|)\,\mathrm{d}x\,\mathrm{d}y\nonumber\\
&\geq C_3 M_0^{-1}(\beta_0-\varepsilon)t_n^{\vartheta+1}\omega_n^{\vartheta+1}e^{2\alpha_0t_n^2\omega_n^2}\int_{B_{1/n}}\int_{B_{1/n}}|x|^{b_0}|y|^{b_0}\,\mathrm{d}x\,\mathrm{d}y\nonumber\\
&=C_3 M_0^{-1}(\beta_0-\varepsilon)t_n^{\vartheta+1}\omega_n^{\vartheta+1}e^{2\alpha_0t_n^2\omega_n^2}\frac{2^{2-\mu}\pi^{2}}{(2+b_0)^{2}}\left(\frac{1}{n}\right)^{4+2b_0-\mu}\nonumber\\
&=C_3 M_0^{-1}(\beta_0-\varepsilon)\frac{2^{2-\mu}\pi^{2}}{(2+b_0)^{2}}t_n^{\vartheta+1}\left(\frac{\sqrt{\log n}}{\sqrt{2\pi}\sqrt{1+\delta_n}}\right)^{\vartheta+1}e^{\alpha_0\pi^{-1}t_n^2(1+\delta_n)^{-1}\log n}\left(\frac{1}{n}\right)^{4+2b_0-\mu}.
\end{align*}
Thus, we deduce
\begin{align}\label{bala4-nivel}
t_n^{2}\geq\ & C_3 M_0^{-1}(\beta_0-\varepsilon)\frac{2^{2-\mu}\pi^{2}}{(2+b_0)^{2}}\left(2\pi(1+\delta_n)\right)^{-\frac{\vartheta+1}{2}}t_n^{\vartheta+1}\nonumber\\
&\times e^{\left[\alpha_0\pi^{-1}t_n^2(1+\delta_n)^{-1}-(4+2b_0-\mu)\right]\log n+\frac{\vartheta+1}{2}\log\left(\log n\right)},
\end{align}
which yields that $(t_n)$ is bounded. Indeed, it follows from \eqref{bala4-nivel} that
\begin{align*}
t_n^{2}\geq \log t_n^{2}\geq & \log \left[C_3 M_0^{-1}(\beta_0-\varepsilon)\frac{2^{2-\mu}\pi^{2}}{(2+b_0)^{2}}\left(2\pi(1+\delta_n)\right)^{-\frac{\vartheta+1}{2}}\right]\\
&+\left[\alpha_0\pi^{-1}t_n^2(1+\delta_n)^{-1}-(4+2b_0-\mu)\right]\log n,
\end{align*}
and consequently
\begin{align}\label{bala5-nivel}
1\geq\ & \frac{1}{t_n^{2}}\log\left[C_3 M_0^{-1}(\beta_0-\varepsilon)\frac{2^{2-\mu}\pi^{2}}{(2+b_0)^{2}}\left(2\pi(1+\delta_n)\right)^{-\frac{\vartheta+1}{2}}\right]\nonumber\\
&+\left[\alpha_0\pi^{-1}(1+\delta_n)^{-1}-t_n^{-2}(4+2b_0-\mu)\right]\log n.
\end{align}
Thus, if $t_n\rightarrow+\infty$ as $n\rightarrow+\infty$, the right hand side of \eqref{bala5-nivel} goes to infinity when $n\rightarrow+\infty$, which leads a contradiction. Thereby, $(t_n)$ is bounded. Hence, passing to a subsequence if necessary, we infer that
$$
\lim_{n\rightarrow+\infty}t_n^{2}=\frac{4-\mu}{\alpha_0}\left(1+\frac{2b_0}{4-\mu}\right)\pi.
$$
Furthermore, as a by-product of \eqref{bala4-nivel}, there exists $C_4>0$ such that
$$
\left[\alpha_0\pi^{-1}t_n^2(1+\delta_n)^{-1}-(4+2b_0-\mu)\right]\log n+\frac{\vartheta+1}{2}\log\left(\log n\right)\leq C_4
$$
which can be rewritten as 
$$
\left[\frac{\alpha_0\pi^{-1}t_n^2-(4+2b_0-\mu)}{1+\delta_n}-\frac{\delta_n(4+2b_0-\mu)}{1+\delta_n}\right]\log n+\frac{\vartheta+1}{2}\log\left(\log n\right)\leq C_4
$$
and by using \eqref{bala1-nivel} we conclude 
$$
-\frac{(4+2b_0-\mu)}{1+\delta_n}\delta_n\log n+\frac{\vartheta+1}{2}\log\left(\log n\right)\leq C_4.
$$
Since, by Lemma \ref{soledade}, $\delta_n\log n\rightarrow 2M_1/(a_0+2)^3$ and $\log(\log n)\rightarrow +\infty$, we obtain a contradiction and this concludes the proof.
\end{proof}

Next we prove Theorem \ref{main-thm2}. For this aim, we will check that the limit of a Palais-Smale sequence for $J$ yields a weak solution to \eqref{P}.

\begin{proof}[Proof of Theorem \ref{main-thm2}]
From the mountain-pass theorem without Palais-Smale condition \cite{Ambro-Rabi} there exists a (PS)$_{c^{\star}}$ sequence $(v_n)\subset Y_{\mathrm{rad}}$ for $J$, i.e.,
$$
\frac{1}{2}\|v_n\|^2-\dfrac{1}{2}\int_{\mathbb{R}^2}\left[|x|^{-\mu}*(Q(|x|)F(v_n))\right]Q(|x|)F(v_n)\,\mathrm{d}x\rightarrow c^{\star}
$$
as well as
$$
\|v_n\|^2-\int_{\mathbb{R}^2}\left[|x|^{-\mu}*(Q(|x|)F(v_n))\right]Q(|x|)f(v_n)v_n\,\mathrm{d}x\rightarrow 0,
$$
as $n\rightarrow+\infty$. From Lemma \ref{limitacao-de-seq-ps}, $(v_n)$ is bounded in $Y_{\mathrm{rad}}$, then, up to  a subsequence, there exists $ v_\star \in Y_{\mathrm{rad}} $ such that $ v_n \rightharpoonup v_\star $ weakly in $ Y_{\mathrm{rad}}$. Moreover, from the boundedness of $(v_n)$ and the above convergences, we have
\begin{equation} \label{E:CC}
\int_{\mathbb{R}^2}\left[|x|^{-\mu}*(Q(|x|)F(v_n))\right]Q(|x|)F(v_n)\,\mathrm{d}x \leq C,\ \ \int_{\mathbb{R}^2}\left[|x|^{-\mu}*(Q(|x|)F(v_n))\right]Q(|x|)f(v_n)v_n\,\mathrm{d}x\leq C,
\end{equation}
with the positive constant $C$ independent of $n$. Next, we split the proof in some claims.

\begin{Claim}\label{af1}
$$
\int_{\mathbb{R}^2}\left[|x|^{-\mu}*(Q(|x|)F(v_n))\right]Q(|x|)F(v_n)\,\mathrm{d}x\rightarrow\int_{\mathbb{R}^2}\left[|x|^{-\mu}*(Q(|x|)F(v_\star))\right]Q(|x|)F(v_\star)\,\mathrm{d}x,
$$
as $n\rightarrow+\infty$.
\end{Claim}
\noindent\textbf{Verification:} We have seen soon after Remark \ref{obs} that if $v_\star\in Y_{\mathrm{rad}}$, then 
$$
\left[|x|^{-\mu}*(Q(|x|)F(v_\star))\right]Q(|x|)F(v_\star)\in L^{1}(\mathbb{R}^{2}).
$$
Thus,
$$
\lim_{M\rightarrow+\infty}\int_{\{|v_\star|\geq M\}}\left[|x|^{-\mu}*(Q(|x|)F(v_\star))\right]Q(|x|)F(v_\star)\,\mathrm{d}x=0.
$$
Let $C>0$ and $\vartheta\in (0,1], M_0, s_0>0$ be the constants given by \eqref{E:CC} and $(f_4)$ respectively. From the above limit, for any $\delta>0$ given, we can choose $M>\max\{ \left(CM_0/\delta\right)^{\frac{1}{\vartheta+1}},s_0\}$ such that
$$
0\leq \int_{\{|v_\star|\geq M\}}\left[|x|^{-\mu}*(Q(|x|)F(v_\star))\right]Q(|x|)F(v_\star)\,\mathrm{d}x < \delta.
$$
By using \eqref{E:CC}, $(f_4)$ and the choice on $M$, we derive
\begin{align*}
0& \leq \int_{\{|v_n|\geq M\}}\left[|x|^{-\mu}*(Q(|x|)F(v_n))\right]Q(|x|)F(v_n)\,\mathrm{d}x \\
&\leq M_0/M^{\vartheta+1} \int_{\{|v_n|\geq M\}}\left[|x|^{-\mu}*(Q(|x|)F(v_n))\right]Q(|x|)v_n f(v_n)\,\mathrm{d}x <\delta.
\end{align*}
This way, for any $\Omega\subset\subset \mathbb R^2$, we have
\begin{multline*}
\left| \int_{\Omega}\left[|x|^{-\mu}*(Q(|x|)F(v_n))\right]Q(|x|)F(v_n)\,\mathrm{d}x - \int_{\Omega}\left[|x|^{-\mu}*(Q(|x|)F(v_\star))\right]Q(|x|)F(v_\star)\,\mathrm{d}x
\right|\\
\leq 2\delta
+\left| \int_{\Omega_{M,v_n}}\left[|x|^{-\mu}*(Q(|x|)F(v_n))\right]Q(|x|)F(v_n)\,\mathrm{d}x \right.\\
-\left. \int_{\Omega_{M,v_\star}}\left[|x|^{-\mu}*(Q(|x|)F(v_\star))\right]Q(|x|)F(v_\star)\,\mathrm{d}x
\right|,
\end{multline*}
where $\Omega_{M,v}:=\Omega\cap \{ |v|\leq M\}$, $v\in Y_{\mathrm{rad}}$. At this point, the proof reduces to establishing that
\begin{equation}\label{E:Campina}
\int_{\Omega_{M,v_n}}\left[|x|^{-\mu}*(Q(|x|)F(v_n))\right]Q(|x|)F(v_n)\,\mathrm{d}x \to \int_{\Omega_{M,v_\star}}\left[|x|^{-\mu}*(Q(|x|)F(v_\star))\right]Q(|x|)F(v_\star)\,\mathrm{d}x,
\end{equation}
as $n\to\infty$, for any fixed $M>\max\{ \left(CM_0/\delta\right)^{\frac{1}{\vartheta+1}},s_0\}$. For this,
we begin noticing that as $K\to\infty$, 
\begin{multline*}
\int_{|v_\star|\leq M}\int_{|v_\star|\leq K} \left[ \frac{Q(|y|)F(v_\star(y))}{|x-y|^\mu}\right]\mathrm{d}y~ Q(|x|)F(v_\star(x))\chi_\Omega \,\mathrm{d}x\\
\longrightarrow \int_{|v_\star|\leq M}\left[ |x|^{-\mu}\ast Q(|x|)F(v_\star)\right]~ Q(|x|)F(v_\star)\chi_\Omega \,\mathrm{d}x,
\end{multline*}
where $\chi_\Omega$ is the characteristic function corresponding to the set $\Omega$. Choose $K>\max\left\{ \left(CM_0/\delta\right)^{\frac{1}{\vartheta+1}},s_0\right\}$ such that
$$
\int_{|v_\star|\leq M}\int_{|v_\star|> K} \left[ \frac{Q(|y|)F(v_\star(y))}{|x-y|^\mu}\right]\mathrm{d}y~Q(|x|) F(v_\star(x))\,\mathrm{d}x<\delta.
$$
Considering \eqref{E:CC}, $(f_4)$ and the choice on $K$, it follows that
\begin{align*}
&\int_{|v_n|\leq M}\int_{|v_n|> K} \left[ \frac{Q(|y|)F(v_n(y))}{|x-y|^\mu}\right]\mathrm{d}y~ Q(|x|)F(v_n(x))\chi_\Omega\,\mathrm{d}x \\ 
&\leq M_0/K^{\vartheta+1} \int_{|v_n|\leq M}\int_{|v_n|> K} \left[ \frac{Q(|y|)v_n(y)f(v_n(y))}{|x-y|^\mu}\right]\mathrm{d}y~ Q(|x|)F(v_n(x))\chi_\Omega\,\mathrm{d}x\\
&\leq M_0/K^{\vartheta+1} \int_{|v_n|\leq M}\int_{|v_n|> K} \left[ \frac{Q(|y|)v_n(y)f(v_n(y))}{|x-y|^\mu}\right]\mathrm{d}y~ Q(|x|) F(v_n(x))\,\mathrm{d}x\\
&\leq M_0/K^{\vartheta+1} \int_{\mathbb R^2}\int_{\mathbb R^2} \left[ \frac{Q(|y|)F(v_n(y))}{|x-y|^\mu}\right]\mathrm{d}y~ Q(|x|) v_n(x)f(v_n(x))\,\mathrm{d}x\\
&= M_0/K^{\vartheta+1} \int_{\mathbb{R}^2}\left[|x|^{-\mu}*(Q(|x|)F(v_n))\right]Q(|x|)f(v_n)v_n\,\mathrm{d}x<\delta.
\end{align*}
Therefore, we obtain
\begin{multline*}
\left|\int_{|v_n|\leq M}\int_{|v_n|> K} \left[ \frac{Q(|y|)F(v_n(y))}{|x-y|^\mu}\right]\mathrm{d}y~ Q(|x|)F(v_n(x))\chi_\Omega\,\mathrm{d}x \right. \\ 
- \left.\int_{|v_\star|\leq M}\int_{|v_\star|> K} \left[ \frac{Q(|y|)F(v_\star(y))}{|x-y|^\mu}\right]\mathrm{d}y~Q(|x|) F(v_\star(x))\chi_\Omega\,\mathrm{d}x 
\right|<2\delta.
\end{multline*}
So, \eqref{E:Campina} will be proved if
\begin{multline} \label{E:Puxinana}
\left|\int_{|v_n|\leq M}\int_{|v_n|\leq K} \left[ \frac{Q(|y|)F(v_n(y))}{|x-y|^\mu}\right]\mathrm{d}y~ Q(|x|)F(v_n(x))\chi_\Omega\,\mathrm{d}x \right. \\ 
- \left.\int_{|v_\star|\leq M}\int_{|v_\star|\leq K} \left[ \frac{Q(|y|)F(v_\star(y))}{|x-y|^\mu}\right]\mathrm{d}y~Q(|x|) F(v_\star(x))\chi_\Omega\,\mathrm{d}x 
\right|\to 0,
\end{multline}
as $n\to\infty$, for any fixed $K,M>0$. In fact, 
\begin{multline*}
\int_{|v_n|\leq K} \left[ \frac{Q(|y|)F(v_n(y))}{|x-y|^\mu}\right]\mathrm{d}y ~ Q(|x|)F(v_n(x))\chi_\Omega\,\mathrm{d}x\\
\longrightarrow \int_{|v_\star|\leq K} \left[ \frac{Q(|y|)F(v_\star(y))}{|x-y|^\mu}\right]\mathrm{d}y ~ Q(|x|)F(v_\star(x))\chi_\Omega\,\mathrm{d}x
\end{multline*}
and, by $(f_1)$, we know there exists a constant $C_{M,K}$ depending on $M$ and $K$ such that
$$
\begin{aligned}
\int_{|v_n|\leq K} \left[ \frac{Q(|y|)F(v_n(y))}{|x-y|^\mu}\right]&\mathrm{d}y ~ Q(|x|)F(v_n(x))\chi_\Omega\,\mathrm{d}x\\
&\leq C_{M,K} \int_{|v_n|\leq K} \left[ \frac{Q(|y|)|v_n(y)|^{\frac{4-\mu}{2}}}{|x-y|^\mu}\right]\mathrm{d}y~ Q(|x|)|v_n(x)|^{\frac{4-\mu}{2}}\chi_\Omega\,\mathrm{d}x\\
&\leq C_{M,K} \left[ |x|^{-\mu}*v_n^{\frac{4-\mu}{2}}\right] Q(|x|)|v_n|^{\frac{4-\mu}{2}}\chi_\Omega\,\mathrm{d}x.
\end{aligned}
$$
Since $v_n\to v_\star$ in $L^q_{\textrm{loc}}(\mathbb R^2)$, for all $q\geq 1$, we reach
$$
\left[ |x|^{-\mu}*v_n^{\frac{4-\mu}{2}}\right] Q(|x|)|v_n|^{\frac{4-\mu}{2}}\chi_\Omega\longrightarrow \left[ |x|^{-\mu}*v_\star^{\frac{4-\mu}{2}}\right] Q(|x|)|v_\star|^{\frac{4-\mu}{2}}\chi_\Omega,
$$
as $n\to\infty$, where we have used the Hardy-Littlewood-Sobolev inequality.  Thus, by Lebesgue's Dominated Convergence Theorem, this finally yields \eqref{E:Puxinana}. 

\begin{Claim} 
$v_\star$ is a weak solution to \eqref{P}.
\end{Claim}
\noindent\textbf{Verification:} Let us now prove that the weak limit $v_\star$ yields actually a weak solution to \eqref{P}, namely that
$$
\int_{\mathbb{R}^2}(\nabla v_\star\cdot\nabla\phi+V(|x|)v_\star\phi)\,\mathrm{d}x-\int_{\mathbb{R}^2}\left[|x|^{-\mu}*(Q(|x|)F(v_\star))\right]Q(|x|)f(v_\star)\phi\,\mathrm{d}x=0,
$$
for all $\phi\in C^{\infty}_0(\mathbb{R}^2)$. Since $(v_n)$ is a (PS)$_{c^\star}$ sequence, for all  $\phi\in C^{\infty}_0(\mathbb{R}^2)$ we know that
$$
\int_{\mathbb{R}^2}(\nabla v_n\cdot\nabla\phi+V(|x|)v_n\phi)\,\mathrm{d}x-\int_{\mathbb{R}^2}\left[|x|^{-\mu}*(Q(|x|)F(v_n))\right]Q(|x|)f(v_n)\phi\,\mathrm{d}x=o_n(1).
$$
Since $v_n\rightharpoonup v_\star$, as $n\to\infty$,  in $Y_{\textrm{rad}}$, it is therefore enough to show that
$$
\int_{\mathbb{R}^2}\left[|x|^{-\mu}*(Q(|x|)F(v_n))\right]Q(|x|)f(v_n)\phi\,\mathrm{d}x\to \int_{\mathbb{R}^2}\left[|x|^{-\mu}*(Q(|x|)F(v_\star))\right]Q(|x|)f(v_\star)\phi\,\mathrm{d}x,
$$
for all $\phi\in C^{\infty}_0(\mathbb{R}^2)$.

Let $\Omega$ be any compact subset of $\mathbb R^2$. We claim that there exists a constant $C(\Omega)$ such that
\begin{equation} \label{E:lagoasequinha}
\int_{\Omega}\left[|x|^{-\mu}*(Q(|x|)F(v_n))\right]Q(|x|)\frac{f(v_n)}{1+v_n}\,\mathrm{d}x\leq C(\Omega).
\end{equation}
Indeed, let 
$$
w_n=\frac{\varphi}{1+v_n},~n\geq 1,
$$
where $\varphi$ is a smooth function compactly supported in $\Omega' \supset\Omega$, $\Omega'$ compact, such that $0\leq \varphi\leq 1$ and $\varphi\equiv 1$ in $\Omega$.  Straightforward computations show that
\begin{align*}
\|w_n\|^2&= \int_{\mathbb{R}^2}\left(\left|\nabla w_n\right|^2+V(|x|)w_n^2\right)\mathrm{d}x\\ 
&=\int_{\mathbb{R}^2}\left(\left|\frac{\nabla \varphi}{1+v_n}-\frac{\varphi\nabla v_n}{(1+v_n)^2}\right|^2+V(|x|)\frac{\varphi^2}{(1+v_n)^2}\right)\mathrm{d}x\\ 
&\leq  \int_{\mathbb{R}^2}\left( \frac{|\nabla \varphi|^2}{(1+v_n)^2}-2\varphi\frac{\nabla\varphi\cdot\nabla v_n}{(1+v_n)^3}+\varphi^2\frac{|\nabla v_n|^2}{(1+v_n)^4}+ V(|x|)\varphi^2 \right)\mathrm{d}x\\ 
&\leq 2\left(\|\varphi\|^2+\|v_n\|^2\right),
\end{align*}
which establishes $w_n\in Y_{\textrm{rad}}$. Since $J'(v_n)\to 0$, as $n\to\infty$, we get $$J'(v_n)w_n\leq \tau_n\|w_n\|,$$ where $\tau_n\to 0$, as $n\to \infty$, that is,
\begin{align*}
\int_{\Omega}\left[|x|^{-\mu}*(Q(|x|)F(v_n))\right]Q(|x|)\frac{f(v_n)}{1+v_n}\,\mathrm{d}x&\leq 
\int_{\mathbb R^2}\left[|x|^{-\mu}*(Q(|x|)F(v_n))\right]Q(|x|)f(v_n) w_n\,\mathrm{d}x\\
& \leq \int_{\mathbb R^2}\left( \nabla v_n\cdot \nabla w_n+ V(|x|) v_n w_n\right) \mathrm{d}x + \tau_n \|w_n\|.
\end{align*}
Then,
\begin{align*}
&\int_{\Omega}\left[|x|^{-\mu}*(Q(|x|)F(v_n))\right]Q(|x|)\frac{f(v_n)}{1+v_n}\,\mathrm{d}x\\
&\leq \int_{\mathbb R^2}\left(|\nabla v_n|^2\frac{\varphi}{(1+v_n)^2}+\frac{\nabla v_n\cdot \nabla \varphi}{1+v_n} +V(|x|)v_n\frac{\varphi}{1+v_n}\right)\,\mathrm{d}x+\sqrt{2}\tau_n\left(\|\varphi\|+\|v_n\|\right) \\
&\leq \|\nabla v_n\|_2^2+\|\nabla \varphi\|_2\|\nabla v_n\|_ 2+\int_{\Omega'}  V(|x|)v_n\,\mathrm{d}x +\sqrt{2}\tau_n\left(\|\varphi\|+\|v_n\|\right)\\
&\leq \|\nabla v_n\|_2^2+\|\nabla \varphi\|_2\|\nabla v_n\|_ 2+\|v_n\|\left(\int_{B_R} V(|x|)\,\mathrm{d}x\right)^{\frac{1}{2}}+\sqrt{2}\tau_n\left(\|\varphi\|+\|v_n\|\right),
\end{align*}
where $B_R$ is a ball in $\mathbb R^2$ that contains $\Omega'$ and  $(V0)$ jointly with H\"older's inequality were used in the last step above. Since $(v_n)$  bounded in $Y_{\mathrm{rad}}$, no further work is required for reach which was claimed.

Now, setting
$$
\xi_n=\left[|x|^{-\mu}*(Q(|x|)F(v_n))\right]Q(|x|)f(v_n),~n\geq 1,
$$
we can notice that
\begin{align*}
\int_{\Omega} \xi_n\,\mathrm{d}x &\leq  2 \int_{\{v_n<1\}\cap \Omega}\left[|x|^{-\mu}*(Q(|x|)F(v_n))\right]Q(|x|)\frac{f(v_n)}{1+v_n}\,\mathrm{d}x\\
& ~~+ \int_{\{v_n\geq 1\}\cap \Omega}\left[|x|^{-\mu}*(Q(|x|)F(v_n))\right]Q(|x|)v_n f(v_n)\,\mathrm{d}x\\
 &\leq  2 \int_{\Omega}\left[|x|^{-\mu}*(Q(|x|)F(v_n))\right]Q(|x|)\frac{f(v_n)}{1+v_n}\,\mathrm{d}x\\
& ~~+ \int_{\mathbb R^2}\left[|x|^{-\mu}*(Q(|x|)F(v_n))\right]Q(|x|)v_n f(v_n)\,\mathrm{d}x, ~\forall n\geq 1.
\end{align*}
Combining \eqref{E:lagoasequinha} and \eqref{E:CC}, we obtain
$$
\int_{\Omega} \xi_n\,\mathrm{d}x \leq  2 C(\Omega)+C, ~\forall\ n\geq 1.
$$
Now, consider the sequence of measures $\mu_n$ given by
$$
\mu_n(E)=\int_{E}\xi_n\,\mathrm{d}x,~n\geq 1,
$$
for each measurable set $E$. Since $\|\xi_n\|_1\leq C_1(\Omega)$, for all $n\geq 1$ and $\Omega$ is bounded, the measures $\mu_n$ have uniformly bounded total variation. Then, by $\text{weak}^*$-compactness, passing to a subsequence, there exists a measure $\mu$ such that $\mu_n{\rightharpoonup}^* \mu$, that is,
$$
\int _\Omega \xi_n\phi\,\mathrm{d}x\to \int _\Omega \phi\,\mathrm{d}\mu, ~n\to\infty,~\forall \phi\in C^{\infty}_0(\Omega).
$$
But then, since $(v_n)$ is a (PS)$_{c^\star}$ sequence, it follows that
$$
\int _{\mathbb R^2} \left( \nabla v_n\cdot \nabla \phi + V(|x|)v_n\phi  \right)\mathrm{d}x \to \int _\Omega \phi\,\mathrm{d}\mu, ~n\to\infty,~\forall \phi\in C^{\infty}_0(\Omega),
$$ 
which implies $\mu$ is absolutely continuous with respect to the Lebesgue measure. So, by using the Radon-Nikod\'ym Theorem, we get  $\xi\in L^1(\Omega)$ such that
$$
\int _\Omega \phi\,\mathrm{d}\mu=\int _\Omega \phi\xi\,\mathrm{d}x,~\forall \phi\in C^{\infty}_0(\Omega).
$$
Since this holds for any compact set $\Omega\subset\mathbb R^2$, we have that there exists a function
$\xi\in L^1_{\textrm{loc}}(\mathbb R^2)$ such that
$$
\int _{\mathbb R^2} \phi\,\mathrm{d}\mu=\lim_{n\to\infty} \int_{\mathbb R^2} \left[|x|^{-\mu}*(Q(|x|)F(v_n))\right]Q(|x|)f(v_n)\,\mathrm{d}x= \int _{\mathbb R^2} \phi\xi\,\mathrm{d}x,~\forall \phi\in C^{\infty}_0(\mathbb R^2).
$$
Moreover, due to $v_n\to v_\star$ a.e. in $\mathbb R^2$,
 $$
 \xi= \left[|x|^{-\mu}*(Q(|x|)F(v_\star))\right]Q(|x|)f(v_\star)
 $$
 and the proof is just finished.

\begin{Claim}  
$v_\star$ is nontrivial.
\end{Claim}
\noindent\textbf{Verification:} Let us suppose by contradiction that $v_\star\equiv 0$. Then, by \eqref{af1},
$$
\lim_{n\rightarrow\infty}\int_{\mathbb{R}^2}\left[|x|^{-\mu}*(Q(|x|)F(v_n))\right]Q(|x|)F(v_n)\,\mathrm{d}x=0.
$$
Thus, since
$$
J(v_n)=\frac{1}{2}\|v_n\|^2-\frac{1}{2}\int_{\mathbb{R}^2}\left[|x|^{-\mu}*(Q(|x|)F(v_n))\right]Q(|x|)F(v_n)\,\mathrm{d}x=c^{\star}+o_n(1),
$$
we get
\begin{equation}\label{cajazeiras}
\lim_{n\rightarrow\infty}\|v_n\|^2=2c^{\star}>0.
\end{equation}
From this and \eqref{est-nivel}, there exists $\delta>0$ so that 
$$
\|v_n\|^2< \frac{4-\mu}{\alpha_0}\left(1+\frac{2b_0}{4-\mu}\right)\pi-\delta
$$
 for $n\in\mathbb{N}$ large. Hence, for $q>1$ close to $1$ and $\alpha>\alpha_0$ close to $\alpha_0$ we will still have for some $\hat{\delta}>0$ that
	$$
	\frac{4q\alpha}{4-\mu}\|v_n\|^{2}\leq \left(1+\frac{2b_0}{4-\mu}\right)4\pi-\hat{\delta},\ \ \ \forall\ n> n_0.
	$$
Thereby, we may proceed analogously to \eqref{tec6} to get
$$
\int_{\mathbb{R}^2}\left[|x|^{-\mu}*(Q(|x|)F(v_n))\right]Q(|x|)f(v_n)v_n\,\mathrm{d}x\rightarrow0,\ \textrm{as}\ n\rightarrow+\infty.
$$
Consequently, $\|v_n\|^2\rightarrow 0,$
as $n\rightarrow+\infty$, which is a contradiction with \eqref{cajazeiras}. Therefore,
$v_\star$ is a nontrivial weak solution of \eqref{P} and this completes the proof of the theorem.
\end{proof}





\bibliographystyle{elsarticle-num}

\end{document}